\documentclass[a4paper,twoside]{amsart}

\usepackage[utf8]{inputenc}
\usepackage{amsmath}
\usepackage{amsfonts}
\usepackage{amsthm}
\usepackage{amssymb}
\usepackage{mathtools}
\usepackage{tikz}
\usepackage[all,cmtip]{xy}
\usepackage{mathrsfs}
\usepackage{color}
\usepackage{enumitem}
\usetikzlibrary{matrix,arrows,calc,intersections}

\theoremstyle{plain}
\newtheorem{theorem}{Theorem}

\newtheorem*{corollary*}{Corollary}
\newtheorem{prop}[theorem]{Proposition}
\newtheorem{lemma}[theorem]{Lemma}
\newtheorem{corollary}[theorem]{Corollary}

\theoremstyle{definition}

\newtheorem{definition}[theorem]{Definition}
\newtheorem*{definition*}{Definition}

\newtheorem{example}[theorem]{Example}

\newtheorem{algorithm}{Algorithm}

\numberwithin{theorem}{section}
\renewcommand{\thealgorithm}{\Alph{algorithm}}

\theoremstyle{remark}

\newtheorem*{remark}{Remark}

\DeclareMathOperator{\Spec}{Spec}

\DeclareMathOperator{\Hom}{Hom}

\DeclareMathOperator{\HH}{H}

\DeclareMathOperator{\QCoh}{QCoh}

\DeclareMathOperator{\Stab}{Stab}

\DeclareMathOperator{\mult}{mult}

\DeclareMathOperator{\cone}{Cone}
\DeclareMathOperator{\rad}{rad}

\DeclareMathOperator{\Bl}{Bl}

\newcommand{\cat}[1]{\ensuremath{\mathrm{#1}}}

\newcommand{\sheaf}[1]{\ensuremath{\mathcal{#1}}}

\newcommand{\Nbl}[1]{\ensuremath{\sheaf{N}_{#1}}}
\newcommand{\perf}[1]{\ensuremath{\mathrm{Perf}(#1)}}
\newcommand{\coh}[1]{\ensuremath{\mathrm{Coh}(#1)}}
\renewcommand{\div}[0]{\ensuremath{\mathrm{div}}}
\newcommand{\res}[0]{\ensuremath{\mathrm{res}}}
\newcommand{\red}[0]{\ensuremath{\mathrm{red}}}
\newcommand{\coarse}[0]{\ensuremath{\mathrm{cs}}}
\newcommand{\csh}[0]{\ensuremath{\mathrm{sh}}}
\newcommand{\gerbe}[1]{\ensuremath{\mathscr{#1}}}
\newcommand{\can}[0]{\ensuremath{\mathrm{can}}}

\newcommand{\Lpb}[1]{\ensuremath{\mathbf{L}{#1}^*}}

\newcommand{\sHom}[0]{\ensuremath{\mathcal{H}om}}

\newcommand{\sHH}[0]{\ensuremath{\mathscr{H}}}

\newcommand{\catset}[1]{\cat{Set}}

\newcommand{\term}[1]{{\em #1}}

\newcommand{\ZZ}[0]{\ensuremath{\mathbb{Z}}}
\newcommand{\VV}[0]{\ensuremath{V}}
\newcommand{\NN}[0]{\ensuremath{\mathbb{N}}}

\newcommand{\RR}[0]{\ensuremath{\mathbb{R}}}
\newcommand{\CC}[0]{\ensuremath{\mathbb{C}}}
\newcommand{\GF}[1]{\ensuremath{\mathbb{F}_{#1}}}

\newcommand{\GG}[0]{\ensuremath{\mathbb{G}}}
\newcommand{\GGm}[0]{\ensuremath{\mathbb{G}_\mathrm{m}}}

\renewcommand{\AA}[0]{\ensuremath{\mathbb{A}}}
\newcommand{\BB}[0]{\ensuremath{\mathrm{B}}}

\newcommand{\KK}[0]{\ensuremath{\mathrm{K}}}

\newcommand{\bs}[1]{\ensuremath{\boldsymbol{#1}}}

\begin{document}
\begin{abstract}
We give an algorithm for removing stackiness from smooth, tame Artin stacks
with abelian stabilisers by repeatedly applying stacky blow-ups. The
construction works over a general base and is functorial with
respect to base change and compositions with gerbes and smooth, stabiliser
preserving maps.
As applications, we indicate how the result can be used for destackifying
general Deligne--Mumford stacks in characteristic zero, and to obtain a weak
factorisation theorem for such stacks. Over an arbitrary field, the method can
be used to obtain a functorial algorithm for desingularising varieties with
simplicial toric quotient singularities, without assuming the presence of a
toroidal structure.
\end{abstract}
\title[Functorial destackification]{Functorial destackification of tame
stacks\\
with abelian stabilisers}
\author{Daniel Bergh}
\date{}
\maketitle
\setcounter{tocdepth}{1}
\tableofcontents
\section{Introduction and main theorems}
Consider an algebraic stack $X$, which is smooth over a field $k$. If $X$ has finite
inertia, then there is a canonical map $X \to X_{\coarse}$ to a coarse (moduli)
space. The algebraic space $X_{\coarse}$ will, however, in general not be
smooth. Given a morphism $f\colon X' \to X$ of stacks with coarse spaces, we
get an induced map $f_\coarse\colon X'_\coarse \to X_{\coarse}$. If $f$ is
proper and birational, we call $f$ a \term{stacky modification}. Our goal is to
find nice choices of $f$ and $X'$ such that map $f_\coarse$ becomes a
desingularisation.

The stacky modifications we will work with are usual \term{blow-ups} with smooth
centres and \term{root stacks}, where we take roots of smooth divisors. Such
modifications will collectively be referred to as \term{stacky blow-ups} with
smooth centres, and sequences of such stacky blow-ups will be referred to as
\term{smooth stacky blow-up sequences} (see Definition~\ref{def-blow-sequence}).

It is useful to think of the process described above as a process to remove
stackiness from a smooth stack. The method described in this paper will produce
a roof-shaped diagram
$$
\xymatrix{
& X' \ar[dl]_{\pi} \ar[dr]^f & \\
X'_\coarse & & X\\
}
$$
where $\pi$ is the coarse map. The map $f$ will be a composition of a sequence
of stacky blow-ups and $\pi$ will be a root stack if we start with an
orbifold $X$ and a composition of a gerbe and a root stack otherwise.
We will use the term \term{destackification} (see
Definition~\ref{def-destackification}) for a process producing such a roof.

In this paper, we will focus on the case when $X$ has diagonalisable
stabilisers. This allows us to attack the problem with toric methods. The
combinatorial nature of toric methods makes them quite insensitive to assumptions
on the base we are working over. Hence, we will assume that the base is an
arbitrary scheme rather than a field. In fact, we could just as easily work
over an arbitrary algebraic stack if we used the appropriate relative versions
of concepts such as coarse space and stabilisers, but we will not work in this
generality.

Just as in the classical method for desingularisation by Hironaka
\cite{hironaka1964}, divisors with simple normal crossings will play an
important role in the algorithms used in this paper. Typically, the divisors
will be produced as exceptional divisors for the various blow-ups used during
the destackification process. As in Hironaka's method,
it will be crucial to keep track of the order in which the divisors have been
created in order to achieve functoriality. The main object that we will work
with will therefore be a pair $(X, \bs{E})$, where $X$ is a tame, smooth stack
and $\bs{E}$ will be an ordered set of smooth divisors on $X$ which have simple
normal crossings. For brevity, we will call such a pair a \term{standard pair}
(see Definition~\ref{def-standard} for technical details). The elements
of $\bs{E}$ will be called the \term{components} of $\bs{E}$.

The first step in the destackification process is to create enough
components of the divisor $\bs{E}$ to be able to attack the problem with toric
methods. We do this by making the pair $(X, \bs{E})$ \term{divisorial} (see
Definition~\ref{def-divisorial-index}).
The reader should be warned that the term \term{divisorial} in this context is
used in a non-standard way. If $X$ is an orbifold, divisoriality has the following
geometric interpretation: each component of $\bs{E}$ is associated to a line bundle, which
in turn is associated to a frame bundle. The pair $(X, \bs{E})$ is \term{divisorial} precisely
when the fibre product of these frame bundles over $X$ is an algebraic space.

\begin{theorem}[Functorial divisorialification]
\label{theoremDivisorial}
Let $(X, \bs{E})/S$ be a standard pair, as defined in Definition~\ref{def-standard}.
If $X$ has diagonalisable stabilisers, then there exists a smooth, ordinary
blow-up sequence
$$
\Pi\colon (X_n, \bs{E}_n) \to \cdots \to (X_0, \bs{E}_0) = (X, \bs{E})
$$
such that the pair $(X_n, \bs{E}_n)$ is divisorial. The construction is
functorial with respect to arbitrary base change $S' \to S$ and with respect
to gerbes and smooth, stabiliser preserving maps $X' \to X$.
\end{theorem}

In \cite{kempf1973} a combinatorial method for desingularising locally toric
varieties is described. This method could quite easily be adapted to handle
destackification of smooth stacks with diagonalisable stabilisers. However, the
method requires a \term{toroidal} structure on the variety. Although the concept
of toroidality extends directly to algebraic stacks (see
Definition~\ref{def-toroidal-index}), it seems non-trivial to obtain such
a structure if not given one from the start. \term{Toroidality} is a
much stronger property than the \term{divisoriality} described above, and
whereas divisorialification may be reached via the naivest possible
method using just ordinary blow-ups (see
Algorithm~\ref{alg-divisorialification}), toroidalification requires the whole
arsenal of stacky blow-ups. In fact, it seems like the easiest way to obtain a
toroidal structure is to simultaneously achieve destackification.

The method described in this paper makes use of two different invariants
associated to each point of the stack. The \term{independency index} (see
Definition~\ref{def-independency-index}) measures how far the stack is from
being destackified at the point and the \term{toroidal index} (see
Definition~\ref{def-toroidal-index}) measures how far the stack is from being
toroidal. The destackification process alternates between reducing the
toroidal index and the independency index in a controlled way. A complication
is that the locus where the toroidal index is maximal is not smooth in general,
and therefore can not be blown-up. Instead other invariants must be used to
single out suitable substacks for modification. The result of the process is
summarised in the following theorem, which is the main theorem of the article.

\begin{theorem}[Functorial destackification]
\label{theoremDestack}
Let $(X, \bs{E})/S$ be a standard pair, as defined in Definition~\ref{def-standard},
over a quasi-compact scheme $S$. If $X$ has diagonalisable stabilisers, then there exists
a smooth, stacky blow-up sequence
$$
\Pi\colon (X_m, \bs{E}_m) \to \cdots \to (X_0, \bs{E}_0) = (X, \bs{E}).
$$
which is a \term{destackification} as in Definition~\ref{def-destackification}.
In particular, the coarse space of $X_m$ is smooth, and the coarse map can be
factored as a gerbe followed by a root stack. The construction is functorial
with respect to arbitrary base change $S' \to S$ and with respect to gerbes and
smooth, stabiliser preserving maps $X' \to X$.
\end{theorem}

\subsection*{Applications}
To illustrate how the destackification theorem
may be applied, we will study three corollaries. The proofs given here will be sketchy,
since a more detailed account will appear later in a joint paper with David Rydh.

The destackification algorithm is useful even if one is not primarily interested
in stacks. Let $X$ be a variety over a field $k$ whose singular points
are all simplicial, toric singularities. By this we mean that each
point $\xi \in X$ has an étale neighbourhood $X' \to X$ with $X' = U/\Delta$ for
some smooth variety $U$ and finite diagonalisable group $\Delta$. In this situation,
there exists a canonical stack $X_\can$ which is smooth and has $X$ as coarse
space \cite{vistoli1989, satriano2012}. By applying the functorial
destackification algorithm on $X_\can$, we obtain a functorial desingularisation
algorithm.
\begin{corollary}[Functorial desingularisation of simplicial toric
singularities]
Let $X$ be an algebraic space of finite type over an arbitrary field $k$. Assume
that $X$ has simplicial toric singularities only. Then there exists a sequence
$$
\Pi\colon X_m \to \cdots \to X_0 = X
$$
of proper birational modifications such that $X_m$ is smooth. The construction
is functorial with respect to change of base field and with respect to smooth
maps $X' \to X$.
\end{corollary}
Note that no toroidal structure is needed. This makes the corollary more
general, than the toroidal methods described in \cite{kempf1973}. On the other
hand, the methods described in this article are somewhat less explicit.

At first sight, the assumption in Theorem~\ref{theoremDestack} that the stack
$X$ has diagonalisable stabilisers seems to be quite restrictive. But at least
if we work over a field of characteristic 0, this can be overcome.
By first using functorial embedded desingularisation on the stacky locus of $X$ with
the Bierstone--Milman variant of Hironaka's method \cite{bm1997}, we reduce
to the case when the stacky locus is contained in a simple normal crossings
divisor. But this implies that the stabilisers are in fact diagonalisable
\cite[Thm.\ 4.1]{ry2000}, so we are in a situation where we can apply
Theorem~\ref{theoremDestack}.

\begin{corollary}[Functorial destackification of Deligne--Mumford stacks in characteristic 0]
\label{cor-destack-zero}
Let $X$ be a Deligne--Mumford stack, which is smooth and of finite type over a
field of characteristic 0. Also assume that $X$ has finite inertia. Then there
exists a smooth stacky blow-up sequence $\Pi$, as in the functorial
destackification theorem, such that $(X_m, E_m)$ has the same properties
as mentioned in that theorem.
\end{corollary}

Finally, destackification can be used to obtain a version of the weak
factorisation theorem by W{\l}odarczyk \cite{wlodarczyk2000} for
Deligne--Mumford stacks in characteristic 0. The corollary is obtained by
applying W{\l}odarczyk's result on the algebraic space obtained after
destackifying using Corollary~\ref{cor-destack-zero}.

\begin{corollary}[Weak factorisation of orbifolds in characteristic 0]
Consider a proper birational map $f\colon X \dashrightarrow Y$ of orbifolds
over a field of characteristic 0. Then there exists a factorisation of $f$ in stacky
blow-ups and blow-downs which is an isomorphism over the non-stacky locus
where $f$ is an isomorphism.
\end{corollary}

\subsection*{Outline of the paper}
Section~\ref{sec-prel} collects some preliminaries on algebraic stacks and
clarifies the terminology used in this paper. We will also make precise definitions
of certain terms, such as functoriality and
blow-up sequence, used in the main theorems. In Section~\ref{sec-toric-stacks}
we will review some basic facts about toric stacks. These will be used in
Section~\ref{sec-toric-destack} where we describe two algorithms,
Algorithm~\ref{alg-part-toric} and~\ref{alg-toric-destack}, which
prove the destackification theorems in the toric case. The algorithms are based
on the classic toric desingularisation algorithm, but have an additional twist
in order to make the process functorial.

From Section~\ref{sec-local-homogeneous} and onwards, we leave the realm of
toric stacks and work with more general smooth stacks with finite diagonalisable
stabilisers. First we show that any such stack is locally toric, which allows us
to work with local homogeneous coordinates.
Then, in Section~\ref{sec-conormal}, we introduce an invariant, which we call
the \term{conormal representation}. This invariant captures the local structure of
a stack near each point. In characteristic 0, we could have worked with the
canonical action of the stabiliser on the tangent space at each point, but in positive
characteristic, the tangent space is not well behaved. Instead, we work with the
conormal bundle of the residual gerbe. We will also study a framework for constructing
special purpose invariants, called \term{conormal invariants}, based on the
conormal representation. Simple, well-known examples of such invariants are the
order of the stabiliser and the multiplicity of the toric singularity of the
corresponding point in the coarse space.

In Section~\ref{sec-alg-outline} we give an outline of the general
destackification algorithms and introduce all conormal invariants used by these
algorithms. Finally, in Section~\ref{sec-general-destack}, we go through the
actual destackification algorithms and prove their correctness.

The paper also includes two appendices, collecting results of more general
interest. In Appendix~\ref{appendix-tame} we prove a structure theorem for
smooth tame stacks in the spirit of the general structure theorem given in
\cite{aov2008}. We will also simplify parts of the proof of the general
structure theorem given in {\em loc.\ cit.} In
Appendix~\ref{appendix-cotangent}, we compute the cotangent complex of a basic
toric stack, and in Appendix~\ref{appendix-cotangent-conormal} we give an
alternative interpretation of the conormal representation in terms of the
cotangent complex.

\subsection*{Acknowledgements}
This project was suggested to me by my advisor, David Rydh.
I am truly grateful for his guidance, enthusiasm and tireless support.

\section{Stacky blow-up sequences and functoriality}
\label{sec-prel}
\subsection{Preliminaries and basic terminology}
We will use the definitions of algebraic stack and algebraic space used
in the Stacks Project \cite{stacks-project}. By a \term{sheaf}, we mean a sheaf
on the site of schemes with the fppf topology, and by a \term{stack}, we mean a stack
in groupoids over the same site. An \term{atlas} for a stack $X$ is a 1-morphism
$f\colon U \to X$, where $U$ and $f$ are representable by algebraic spaces and $f$
is flat and locally of finite presentation. If the morphism $f$ is smooth,
we call it a \term{smooth atlas}. A stack is \term{algebraic} if it admits an atlas,
and it is a theorem that every algebraic stack admits a smooth atlas.

Let $X$ be an algebraic stack. A morphism $\pi\colon X \to X_\coarse$ is called
a \term{coarse space} if it is initial among morphisms to algebraic spaces and
the induced map $|\pi|\colon |X| \to |X_\coarse|$ between topological spaces is
a homeomorphism. Usually, this is called a \term{coarse moduli space}, but we
drop the word {\em moduli} since we are discussing algebraic stacks without
having any specific moduli problem in mind. Due to a classical theorem by Keel
and~Mori \cite{km1997} with generalisations by Conrad \cite{conrad2005} and
Rydh \cite{rydh2013}, an algebraic stack $X$ has a coarse space if its
inertia stack is finite over $X$.

Let $X$ be an algebraic stack which is quasi-separated and locally of finite
presentation over a base scheme $S$. Following Abramovich, Olsson and
Vistoli~\cite{aov2008}, we say that $X$ is \term{tame} if it has finite inertia
and linearly reductive stabilisers. This property is reviewed in
Appendix~\ref{appendix-tame}. We will be particularly interested in the case
when $X$ has diagonalisable stabilisers. We will use the term \term{orbifold},
in the relative sense, for a tame stack $X \to S$ which is smooth over the
base scheme, and which has fibrewise generically trivial stabilisers.

The usual concept of simple normal crossings divisors generalises directly to
stacks in the relative setting. Let $X \to S$ be a smooth stack over a scheme
and let $E = E^1 + \cdots + E^r$ be an effective Cartier divisor on $X$, with
each $E^i$ smooth over $S$. Note that $E$ is a relative effective Cartier divisor
in the sense of \cite[§21.15]{egaIV4}. Let $F = F^1 + \cdots + F^r$ be the
pull-back of $E$ along a smooth atlas $U \to X$. We say that $E$ has
\term{simple normal crossings} if the fibre $F_\xi \subset U_\xi$ has simple
normal crossings in the usual sense for each geometric point $\xi\colon\Spec k
\to S$. {\em Mutatis mutandis}, we define what is meant for a closed
substack $Z \subset X$, which is smooth over $S$, to have simple normal
crossings with $E$.
 
\begin{definition}
\label{def-standard}
Let $S$ be a scheme and consider a pair $(X, \boldsymbol{E})/S$, where
\begin{enumerate}
\item $X$ is a tame algebraic stack which is smooth and of finite presentation
over $S$.
\item $\boldsymbol{E} = (E^1, \ldots, E^r)$ is an ordered set of distinct,
effective Cartier divisors on $X$, called the \term{components} of $\boldsymbol{E}$.
Each component $E^i$ is required to be smooth over $S$ and their sum
$E = \sum E^i$ is required to be a simple normal crossings divisor.
\end{enumerate}
We call such a pair $(X, \boldsymbol{E})/S$ a \term{standard pair}.
\end{definition}

\noindent
Note that the term \term{component} in this context does not refer to
\term{connected component}; the components of $\boldsymbol{E}$, as in the
definition above, may well be empty or disconnected.

When referring to the ordering of the components of an ordered simple normal
crossings divisor, we will use an age metaphor. The components of such a divisor
form a sequence $E^1, \ldots, E^r$. The indices may be thought of as birth dates
of the components, and we say that $E^i$ is \term{older} than $E^j$, and that
$E^j$ is \term{younger} than $E^i$ provided that $i < j$.

\subsection{Stacky blow-up sequences}
Let $S$ be a scheme and $(X, \boldsymbol{E})/S$ be a standard pair. By a
\term{smooth blow-up} of $(X, \boldsymbol{E})/S$, we mean a blow-up
$\pi\colon\Bl_Z X \to X$ in a centre $Z$ which is smooth over $S$ and having
simple normal crossings only with $E$. The \term{transform} of
$(X, \boldsymbol{E})/S$ along $\pi$ is the pair $(\Bl_Z X, \boldsymbol{E}')$,
where $\boldsymbol{E}'$ denotes the ordered set of the strict transforms of the
components of $\boldsymbol{E}$ followed by the exceptional divisor of the
blow-up.

The \term{root construction} of a stack in an effective Cartier divisor is
thoroughly described in for instance \cite{agv2008, cadman2007} and
\cite[§1.3.b.]{fmn2010}. Let $(X, \boldsymbol{E})/S$ be a standard pair.
We will only consider root stacks with roots taken of components of
$\boldsymbol{E}$. Such a root stack will be called a \term{smooth root stack}.
If $E^i \in \boldsymbol{E}$, we use the notation $X_{d^{-1}E^i}
\to X$ for the $d$-th root of $E^i$. If $\boldsymbol{E}' \subset
\boldsymbol{E}$ is a subset of components, and $\boldsymbol{d}$ is a sequence
of positive integers indexed by the elements of $E'$, then
$X_{\boldsymbol{d}^{-1}E'} \to X$ denotes the fibre product of the
stacks $X_{d^{-1}_iE^i}$ over $X$ for all $E^i \in \boldsymbol{E}'$ with $d_i$
as corresponding element in $\boldsymbol{d}$. The pair $(\bs{d}, \bs{E}')$ is
called the \term{centre} of the root stack.

The \term{transform} of $(X, \boldsymbol{E})/S$ along a smooth root stack
$\pi\colon X_{\boldsymbol{d}^{-1}\boldsymbol{E}'} \to X$ is the pair
$(X_{\boldsymbol{d}^{-1}\boldsymbol{E}'}, \pi^{-1}\boldsymbol{E} \cup
\boldsymbol{F})$. Here $\pi^{-1}E$ denotes the set of strict transforms of the
components of $\boldsymbol{E}$, and $\boldsymbol{F}$ is the set of roots
corresponding to the elements in $\boldsymbol{E}'$. The sets
$\pi^{-1}\boldsymbol{E}$ and $\boldsymbol{F}$ inherit their ordering from
$\boldsymbol{E}$ and $\boldsymbol{E}'$ respectively. In the union
$\pi^{-1}\boldsymbol{E}\cup \boldsymbol{F}$, the elements of $\boldsymbol{F}$
are considered younger than the other elements.

Collectively, smooth root stacks and smooth blow-ups are referred to as
\term{smooth stacky blow-ups}. The transform of a stack--divisor pair
satisfying the standard assumptions along a smooth stacky blow-up again
satisfies the standard assumptions.

\begin{definition}
\label{def-blow-sequence}
Let $(X_0, \boldsymbol{E}_0)/S$ be a standard pair. A \term{smooth, stacky
blow-up sequence} of $(X_0, \boldsymbol{E}_0)/S$ of length $n$ is a
sequence
$$
\Pi\colon (X_n, \boldsymbol{E}_n) \stackrel{\pi_r}{\to} \cdots
\stackrel{\pi_1}{\to} (X_0, \boldsymbol{E}_0)
$$
where each $\pi_i$, for $1 \leq i \leq n$, is a smooth stacky blow-up in a
centre $Z_{i-1}$ and each $(X_i, \boldsymbol{E}_i)$ is the transform of
$(X_{i-1}, \boldsymbol{E_{i-1}})$ along $\pi_i$. The centres $Z_{i}$ for $0
\leq i \leq n-1$, although suppressed from the notation, are considered part of
the structure. We require each $Z_i$ to have positive codimension in $X_i$ at
each of its points. If all stacky blow-ups are in fact usual blow-ups, we
call $\Pi$ a \term{smooth, ordinary blow-up sequence}.
\end{definition}

Since all blow-up sequences we consider in this article will be smooth, stacky
blow-up sequences, we will usually drop the modifiers \term{smooth} and
\term{stacky} and just say \term{blow-up sequence}.

\begin{definition}
\label{def-destackification}
Let $(X_0, \boldsymbol{E_0})/S$ be a stack--divisor pair satisfying the standard
assumptions, and $$
\Pi\colon (Y, F) = (X_n, \boldsymbol{E}_n) \stackrel{\pi_r}{\to} \cdots
\stackrel{\pi_1}{\to} (X_0, \boldsymbol{E}_0)
$$
a smooth, stacky blow-up sequence on $(X_0, \boldsymbol{E}_0)/S$. Let $\pi\colon
Y \to Y_\coarse$ be the coarse space. We call $\Pi$ a \term{destackification} if the
following conditions hold:
\begin{enumerate}
\item The space $Y_\coarse$ is smooth over $S$.
\item The components of $\boldsymbol{F}_\coarse = \{F^i_\coarse \mid F^i \in
F\}$ are smooth over $S$ and have simple normal crossings only.
\item The divisor $\boldsymbol{F}$ is a $\boldsymbol{d}$-th root of the
pull-back $\pi^\ast \boldsymbol{F}_\coarse$ for some sequence $\boldsymbol{d}$ of
positive integers indexed by the components of $\boldsymbol{F}$.
\item The canonical factorisation $Y \to
(Y_\coarse)_{\boldsymbol{d}^{-1}\boldsymbol{F}_\coarse} \to Y_\coarse$ through the root stack
makes $Y$ a gerbe over
$(Y_\coarse)_{\boldsymbol{d}^{-1}\boldsymbol{F}_\coarse}$.
In particular, if $X_0$ is an orbifold, then
$Y \to (Y_\coarse)_{\mathbf{d}^{-1}\boldsymbol{F}_\coarse}$
is an isomorphism.
\end{enumerate}
The conditions 1 and 2 can be summarised by saying that the pair
$(Y_\coarse, \boldsymbol{F}_\coarse)/S$ is a standard pair.
\end{definition}

A stacky blow-up is said to be \term{empty} if the centre is empty.
Although the algorithms used in the constructions mentioned in the main theorems
will never produce blow-up sequences containing empty blow-ups, such may
occur after pulling back blow-up sequences along morphisms which are not surjective.
We will consider such pull-backs when discussing functoriality below. We regard
two blow-up sequences $\Pi$ and $\Pi'$ to be \term{equivalent} if, after pruning
them from empty blow-ups, they fit into a 2-commutative ladder
$$
\xymatrix{
\Pi\colon (X_n, \boldsymbol{E}_n) \ar[r] \ar[d] & \ldots \ar[r] & (X_0,
\boldsymbol{E}_0) \ar[d]\\
\Pi'\colon (X'_n, \boldsymbol{E}'_n) \ar[r] & \ldots \ar[r] & (X'_0,
\boldsymbol{E}'_0)\\} $$
such that the vertical morphism are isomorphisms preserving the centres.

\subsection{Gerbes}
Let $\pi\colon X \to Y$ be a morphism of algebraic stacks. We say
that $\pi$ is a \term{gerbe} if $X$ is a gerbe in the topological sense, as defined by
Giraud \cite[Def.\ 2.1.1]{giraud1971}. Here we view $X$ as a stacks in
groupoids over the site $Y$ with the fppf topology inherited from the site of
schemes. We also use the term gerbe in the absolute sense. An algebraic stack
$X$ is a \term{gerbe} if it is a gerbe over an algebraic space.
This way of using the terminology, which is standard and used for instance in
the Stacks Project~\cite{stacks-project}, might occasionally cause some
confusion. For instance, if $\pi\colon X \to Y$ is a gerbe, then $\pi$ is smooth
as a morphism of algebraic stacks for quite elementary reasons
(see Proposition~\ref{prop-gerbe-smooth}). But this does not imply that $X$ is a
gerbe over $Y$ in the topological sense when using the smooth topology on $Y$.

\subsection{Stabiliser preserving maps}
We recall the definition and some basic facts about stabilisers preserving
1-morphisms of stacks.
Let $f\colon X \to Y$ be a 1-morphism of stacks. Given a
generalised point, $\xi\colon T \to X$, where $T$ is a scheme, we get an
induced map of stabilisers $\Stab_\xi X \to \Stab_{f\circ\xi} Y$ over $T$.
The map $f$ also induces a pair of 2-commutative diagrams
$$
\xymatrix{
I_X \ar[r] \ar[d] & I_Y \ar[d] & X \ar[r]^f \ar[d] & Y \ar[d] \\
X \ar[r]_f & Y & X_\csh \ar[r] & Y_\csh. \\
}
$$
Here $I_X \to X$ denotes the inertia stack of $X$, and the map $X \to X_\csh$ is the
\term{coarse sheaf} of $X$, by which we mean the map which is initial among maps
to sheaves.
\begin{definition}
The 1-morphism $f\colon X \to Y$ is called \term{stabiliser preserving} if any of the
following conditions, which are easily seen to be equivalent, hold:
\begin{enumerate}
\item The map $\Stab_\xi X \to \Stab_{f\circ\xi} Y$ is an isomorphism for all generalised
points $\xi$.
\item The left 2-commutative square above is 2-cartesian. 
\item The right 2-commutative square above is 2-cartesian. 
\end{enumerate}
If the first condition holds for all geometric points, we say that $f$ is
\term{point-wise stabiliser preserving}. 
\end{definition}
\noindent
In particular, monomorphisms between stacks are stabiliser preserving. Note that the notions
of \term{stabiliser preserving} and \term{point-wise stabiliser preserving}
maps are distinct.
\begin{example}
Let $k$ be a field and let $X = \Spec k[\varepsilon]$ be the spectrum of the
dual numbers over $k$. Furthermore, we let the group $\mu_2$ act on $X$ by giving
$\varepsilon$ degree 1. Then we get a map
$[X/\mu_2] \to \BB\mu_2$ from the quotient stack to the classifying stack of $\mu_2$, which
is pointwise stabiliser preserving, but not stabiliser preserving. 
\end{example}
A useful fact is that if $f\colon X \to Y$ is an étale map between algebraic
stacks with finite inertia, then the locus where $f$ is point-wise stabiliser
preserving is open in $X$, and $f$ is stabiliser preserving over this locus
\cite[Prop.~6.5]{rydh2011}. In fact, if the stacks are tame, the corresponding
fact for smooth morphisms is also true, but we will not use this here. 

\subsection{Functoriality} We consider two basic situations when a blow-up
sequence can be transferred from one standard pair to another. Fix a standard pair
$(X, \boldsymbol{E})/S$.

The first situation is when we change base scheme. Given a morphism $S' \to S$,
we can form the pull-backs $X' = X\times_S S'$ and $\boldsymbol{E}' =
\boldsymbol{E} \times_S S'$. Then the pair $(X', \boldsymbol{E}')/S'$ also
satisfies the standard assumptions, and any blow-up sequence on $(X,
\boldsymbol{E})/S$ pulls back to a blow-up sequence on $(X',
\boldsymbol{E}')/S'$.

The second situation is when we have a morphism of stacks $X' \to X$ which is
smooth. Then we can form the pull-back $\boldsymbol{E}' = \boldsymbol{E}
\times_X X'$, and we get a pair $(X', \boldsymbol{E}')/S$. Again, any blow-up
sequence on $(X, \boldsymbol{E})/S$ pulls back to a blow-up sequence on $(X',
\boldsymbol{E}')/S$.

We say that a construction of a blow-up sequence is \term{functorial} with
respect to a certain kind of maps, fitting into one of the above situations,
provided that the blow-up sequence obtained from the construction applied to
$(X', \boldsymbol{E}')$ is equivalent to the pull-back of blow-up sequence
obtained from the construction applied to $(X, \boldsymbol{E})$.

The constructions in the main theorems are functorial with respect to arbitrary
pull-backs. It is, however, not reasonable to expect the construction to be
functorial with respect to arbitrary smooth maps as in the second case described
above. Indeed, if we take the morphism $X' \to X$ to be a smooth atlas, we
expect the destackification of $X'$ to be trivial, whereas the destackification
of $X$ should certainly not be trivial in general. But the constructions in
the main theorem are functorial with respect to morphism $X' \to X$ that
preserves stackiness, that is, morphisms which are stabiliser preserving. They
are also insensitive to generic stabilisers in the sense that they are
functorial with respect to gerbes.

\subsection{Distinguished structure}
We do not want our algorithms to modify the locus lying over the smooth locus of the
coarse space of the original stack. This poses a problem when it comes to root stacks,
since they always modify the entire divisor of which the root is taken. Thus, we would
like to keep track of divisors which we are allowed to root. We do this by marking
certain divisors as \term{distinguished}.
\begin{definition}
\label{def-distinguished}
Let $(X, \boldsymbol{E})/S$ be a stack--divisor pair satisfying the standard
assumptions. Let $\boldsymbol{D} \subset \boldsymbol{E}$ be a subset such that
all divisors in $\boldsymbol{D}$ are younger than the divisors in the complement
$\boldsymbol{E}\setminus \boldsymbol{D}$. We say that $(X, \boldsymbol{E},
\boldsymbol{D})/S$ is a stack--divisor pair with \term{distinguished structure},
and call the components of $\boldsymbol{E}$ lying in $\boldsymbol{D}$
\term{distinguished}. A stacky blow-up of $(X, \boldsymbol{E},
\boldsymbol{D})/S$ is called \term{admissible} if the centre is contained in the
support of $\boldsymbol{D}$.
\end{definition}
\noindent
The transform $(X', \bs{E}')/S$ of an admissible stacky blow-up of a
stack--divisor pair with distinguished structure $(X, \bs{E},
\bs{D})/S$, again has a distinguished structure $\bs{D}'$.
This is defined by letting $\bs{D}'$ be the set containing the
exceptional divisor of the stacky blow-up along with the strict transforms of all
distinguished divisors in $\bs{D}$.

\section{Smooth toric stacks}
\label{sec-toric-stacks}
The theory of toric stacks has been treated by several authors. We mention a few.
Borisov, Chen and Smith~\cite{bcs2005} give a basic definition of smooth toric
Deligne--Mumford stacks via the Cox construction. Iwanari gives a moduli interpretation
of toric stacks using logarithmic geometry~\cite{iwanari2009a}. He also gives a structure
theorem, characterising toric orbifolds over a field of characteristic zero in terms of
stacks with torus actions \cite{iwanari2009b}. A similar result is obtained independently
by Fantechi, Mann and Nironi~\cite{fmn2010}, using a bottom up construction. Geraschenko 
and Satriano \cite{gs2011a, gs2011b} extend the theory to non-smooth stacks and stacks
with positive-dimensional stabilisers and unify the theory with other notions of toric
stacks.

In this section, we summarise some of the basic theory of smooth toric stacks
with finite stabilisers. Since this is the only kind of toric stacks we will
consider in this article, we will simply refer to them as \term{toric
stacks}. If in addition, they have trivial generic stabilisers, we call them
\term{toric orbifolds}. We give no proofs of the statements, since they are
either implicitly or explicitly proven in the above references, or can be left
as simple exercises. It should be noted that most of the above
references work over the field of complex numbers, whereas we will work over
an arbitrary base scheme $S$. This, however, does not introduce any extra
complications at this level. Whenever it applies, we follow the notation used
in \cite{cls2011} and \cite{bcs2005}.

\subsection{Basic toric stacks.} 
First we introduce \term{basic toric stacks}. They play the same role in the
theory of toric stacks as affine toric varieties in the theory of toric
varieties. Toric stacks in general are obtained by gluing basic toric stacks
together along toric morphisms in the Zariski topology. Note that the term
basic toric stack is non-standard.

First, we describe a more general class of algebraic stacks. Fix a scheme $S$.
Let $\mathcal{C}$ be the category of pairs $(R, A)$, where $A$ is a finitely
generated abelian group and $R$ a sheaf of $A$-graded $\sheaf{O}_S$-algebras.
A morphism $(R, A) \to (R', A')$ is a group homomorphism $A \to A'$ together with an 
$A'$-graded $\sheaf{O}_S$-algebra homomorphism $R \to R'$, where $R$ receives its
$A'$-grading via the group homomorphism $A \to A'$.

The grading of $A$ on $R$ corresponds to an action of the Cartier dual $A^\vee$ on
$\Spec_{\sheaf{O}_S}R$. This gives us a contravariant functor from $\mathcal{C}$ to
the 2-category of algebraic stacks,
taking $(R, A)$ to $[\Spec_{\sheaf{O}_S}R/A^\vee]$. Given a pair $(R, A)$ in $\mathcal{C}$
such that the corresponding stack $X$ has finite stabiliser, the morphism $(R_0, 0) \to (R, A)$
corresponds to the coarse map $X \to X_\coarse$. Consider a pair
of morphisms $f\colon (R, A) \to (R', A')$, $g\colon (R, A) \to (R'', A'')$, and
assume that the group homomorphism $A \to A'$ underlying $f$ is injective. A useful
fact, which is used in the proof of correctness for Algorithm~\ref{alg-destack},
is that in this situation the push-out square of $f$ and $g$ corresponds to a 2-fibre
product of the corresponding stacks.

\begin{definition}
\label{def-basic-toric}
An algebraic stack $X$ associated to a pair $(R, A)$, as described above, is called
a \term{basic toric stack} provided that the following two conditions hold:
\begin{enumerate}
\item[(a)] The sheaf of rings $R$ is of the form
$
R = \sheaf{O}_S[x_1, \ldots, x_r][x_{s+1}^{-1}, \ldots, x_r^{-1}],
$
for some $r$ and $s$ such that $0\leq s\leq r$.
\item[(b)] Each coordinate function $x_i$, with $1 \leq i \leq r$, is homogeneous
of degree $a_i \in A$.
\end{enumerate}
The triple $(R, A, \bs{a})$, where $\bs{a} = (a_1, \ldots, a_r)$, is called a
\term{homogeneous coordinate ring} for $X$.
The closed substacks of the form $E^i = \VV(x_i)$, for $1 \leq i \leq s$, are called
the \term{toric divisors} of $X$. A morphism of basic toric stacks is called
\term{toric} provided that it comes from a morphism $(R, A) \to (R', A')$ such
that the underlying $\sheaf{O}_S$-algebra homomorphism $R \to R'$ takes monomials
to monomials.

By default, our basic toric stacks will always have finite stabilisers,
but the definition is equally meaningful without this assumption.
\end{definition}

It should be noted that the homogeneous coordinate ring does not determine the
basic toric stack uniquely. For instance, we may always assume that the
weights in the vector $\bs{a}$ corresponding to the coordinates $x_{s+1},
\ldots, x_r$ are zero. Indeed, let $A'$ be the quotient of $A$ by the subgroup
generated by those weights, and let $\bs{a}'$ be the corresponding weight
vector. Then there is a basic toric stack associated to a triple
$(R', A',\bs{a}')$ which is equivalent to the basic toric stack associated to
$(R, A,\bs{a})$. In particular, we can usually simply ignore the coordinates
$x_{s+1}, \ldots, x_r$ in arguments about basic toric stacks, since they just
correspond to a factor by a torus. If $r = s$, we say that the basic toric stack
is without torus factors.
It should also be noted that although the toric divisors of a basic toric stack
are basic toric stacks in their own right, the inclusions into the original
stacks are not toric.

If we order the coordinate functions, then the set $\bs{E}$ of toric divisors
on a basic toric stack $X$ inherits an ordering, and we get a standard pair
$(X, \bs{E})$. Indeed, this kind of standard pair is prototypical, and in
Section~\ref{sec-local-homogeneous} we will see that any standard pair with
diagonalisable stabilisers is locally a basic toric stack.

\subsection{Toric orbifolds.}
As with toric varieties, the gluing together of basic toric
stacks can be described combinatorially. We review the parts of the theory we
need in this article, restricting the discussion to toric orbifolds with no
torus factors.

Let $N$ be a lattice of rank $n$, and consider it as a subset of the vector
space $N_\RR := N\otimes_\ZZ \RR$. By a \term{cone} $\sigma$ in $N_\RR$, we will always
mean a \term{polyhedral}, \term{rational} and \term{strictly convex} cone.
We write $\tau \preceq \sigma$ if $\tau$ is a face of $\sigma$.
By $\sigma(1)$ we mean the set of 1-dimensional faces, also called the
\term{extremal rays}, of $\sigma$. Recall that $\sigma$ is called \term{simplicial}
if the cardinality of $\sigma(1)$ equals the dimension of the subspace of $N_\RR$
spanned by $\sigma(1)$.

Given a \term{fan} $\Sigma$ in $N_\RR$, we denote the set of \term{rays}, that is the
set of 1-dimensional cones, in $\Sigma$ by $\Sigma(1)$. A fan is \term{simplicial} if
all its cones are. We will frequently consider the free abelian group $\ZZ^{\Sigma(1)}$
on the set of rays in a fan $\Sigma$. An element $c_1\rho_1 + \cdots + c_r\rho_r$,
with $c_i \in \ZZ$ and $\rho_i \in \Sigma(1)$, is called \term{effective} if all
coefficients $c_i$ are greater or equal to zero.

\begin{definition}
A \term{stacky fan} is a triple $\mathbf{\Sigma} = (N, \Sigma, \beta)$, where $N$ is a finitely
generated free abelian group, $\Sigma$ is a \term{simplicial} fan in $N_\RR$
such that $|\Sigma|$ spans $N_\RR$,  and $\beta\colon \ZZ^{\Sigma(1)} \to N$ is a
group homomorphism taking each generator $\rho \in \ZZ^{\Sigma(1)}$ to a non-zero
lattice point on the ray $\rho$.
\end{definition}

Given a stacky fan $\mathbf{\Sigma} = (N, \Sigma, \beta)$, we construct a toric orbifold
via the \term{Cox construction}. Denote the dual $\Hom_\ZZ(N, \ZZ)$ by $M$. Then
the Cartier dual of $M$ over $S$ is an $n$-dimensional torus, which we denote by $T_N$. Its
cocharacter and character groups may be canonically identified with $N$ and $M$
respectively. The morphism $\beta$ induces a homomorphism of algebraic groups
$T_{\ZZ^{\Sigma(1)}} \to T_N$, which fits into an exact sequence
$$
1 \to \Delta(\mathbf{\Sigma}) \to T_{\ZZ^{\Sigma(1)}} \to T_N \to 1
$$
where the exactness at the term $T_N$ is ensured by the fact that $|\Sigma|$
spans $N_\RR$. Now consider the lattice $\ZZ^{\Sigma(1)}$ and the
corresponding space $\RR^{\Sigma(1)}$. Given a cone $\sigma \in \Sigma$, we have
a corresponding cone $\widetilde{\sigma}$ in $\RR^{\Sigma(1)}$ spanned by the
rays $\rho \in \sigma(1)$ viewed as generators in $\RR^{\Sigma(1)}$.
Collectively, the cones $\widetilde{\sigma}$ for $\sigma \in \Sigma$, form a
fan $\widetilde{\Sigma}$ in $\RR^{\Sigma(1)}$. Denote the corresponding toric
variety, or rather family of toric varieties over $S$, by
$X_{\widetilde{\Sigma}}$.

We give an explicit description of the family $X_{\widetilde{\Sigma}}$ of
varieties. The \term{total coordinate ring} associated to $\Sigma$ is
the polynomial ring $R = \sheaf{O}_S[x_\rho \mid \rho \in \Sigma(1)]$. The
\term{irrelevant ideal} is the ideal
$$
B(\Sigma) = \langle x^{\hat{\sigma}} \mid \sigma \in \Sigma \rangle,
$$
where $x^{\hat{\sigma}}$ denotes the product of all elements $x_\rho$
with $\rho \not \in \sigma(1)$. Let $\AA^{\Sigma(1)}_S = \Spec_{\sheaf{O}_S} R$
be the relative spectrum and $Z(\Sigma)$ be the closed subscheme associated
to the irrelevant ideal $B(\Sigma)$. The scheme $X_{\widetilde{\Sigma}}$
is simply $\AA^{\Sigma(1)}_S \setminus Z(\Sigma)$. Note that the torus
$T_{\ZZ^{\Sigma(1)}}$ is embedded in $X_{\widetilde{\Sigma}}$ in a natural way,
and the action of $\Delta(\mathbf{\Sigma})$ on $T_{\ZZ^{\Sigma(1)}}$ extends to
$X_{\widetilde{\Sigma}}$.

\begin{definition}[The Cox construction]
\label{def-cox-const}
Let $\mathbf{\Sigma} = (N, \Sigma, \beta)$ be a stacky fan, and consider the
group $\Delta(\mathbf{\Sigma})$ acting on the scheme $X_{\widetilde{\Sigma}}$ over the
base scheme $S$ as defined above. The \term{toric orbifold} $X_{\bs{\Sigma}}$
associated to $\bs{\Sigma}$ is defined as the stack quotient
$[X_{\widetilde{\Sigma}}/\Delta(\bs{\Sigma})]$.
\end{definition}

Just like in the case with usual toric varieties, there is an order reversing
correspondence between cones in $\mathbf{\Sigma} = (N, \Sigma, \beta)$
and orbit closures in $X_{\mathbf{\Sigma}}$. Given a cone $\sigma \in \Sigma$,
we have a closed variety $\VV(\langle x_\rho, \rho \in \sigma(1)\rangle)$
in $X_{\widetilde{\Sigma}}$. Since this closed variety is
$\Delta(\mathbf{\Sigma})$-invariant, it descends to a closed substack $V(\sigma)$ of
$X_{\mathbf{\Sigma}}$. In the particular case when we have a ray
$\rho \in \Sigma(1)$, the substack $V(\rho) \subset X_{\mathbf{\Sigma}}$ is a prime divisor,
and we denote it by $D_\rho$. The divisor $D_\rho$ is a smooth Cartier divisor.
More generally,
if $\psi = c_1\rho_1 + \cdots + c_r\rho_r$ is an element of $\ZZ^{\Sigma(1)}$,
we let $D_\psi$ denote the divisor $c_1D_{\rho_1} + \cdots + c_rD_{\rho_r}$.
Such a divisor is called a \term{toric divisor}, and it has simple normal crossings only.

\subsection{Morphisms of toric orbifolds}
Next we describe morphisms of stacky fans and toric orbifolds.
Our definition is different than, but equivalent to, the one given by
Iwanari in \cite{iwanari2009b}.

Recall that a morphism of fans $f \colon (N, \Sigma) \to (N', \Sigma')$
is a group homomorphism $f\colon N \to N'$ such that the induced map
$f_\RR = f \otimes_\ZZ \RR$ maps each cone $\sigma \in \Sigma$ into a cone
$\sigma' \in \Sigma'$. This extends to stacky fans as follows.
\begin{definition}
Consider the stacky fans $\mathbf{\Sigma} = (N, \Sigma, \beta)$
and $\mathbf{\Sigma}' = (N', \Sigma', \beta')$. A \term{morphism}
$\mathbf{\Sigma} \to \mathbf{\Sigma}'$ of stacky fans is a pair
$(f, \hat{f})$ of group homomorphisms fitting into a commutative square
$$
\xymatrix{
\ZZ^{\Sigma(1)} \ar[r]^{\hat{f}}\ar[d]_{\beta}
& \ZZ^{\Sigma'(1)} \ar[d]^{\beta'} \\
N \ar[r]_{f} & N',
}
$$
such that both $f\colon (N, \Sigma) \to (N', \Sigma')$ and
$\hat{f}\colon (\ZZ^{\Sigma(1)}, \widetilde{\Sigma}) \to
(\ZZ^{\Sigma'(1)}, \widetilde{\Sigma}')$ are morphisms of fans.
Since $\hat{f}$ is uniquely determined by $f$, we often omit $\hat{f}$
from the notation, and simply say that $f\colon\mathbf{\Sigma} \to
\mathbf{\Sigma}'$ is a morphism of stacky fans.
\end{definition}

It is easy to see that a morphism $f\colon\mathbf{\Sigma} \to
\mathbf{\Sigma}'$ of stacky fans induces a corresponding equivariant
morphism of pairs $(X_{\widetilde{\Sigma}}, \Delta(\mathbf{\Sigma})) \to
(X_{\widetilde{\Sigma}'}, \Delta(\mathbf{\Sigma}'))$ which, in turn,
induces a 1-morphism $X_{\mathbf{\Sigma}} \to X_{\mathbf{\Sigma}'}$ of
toric orbifolds. This gives a functor from the category of stacky fans
to the category of orbifolds over a base scheme $S$, and we call its essential
image the \term{category of toric orbifolds} (without torus factors).

The simplest example of a toric morphism is that of toric open immersions,
which correspond to subfans of stacky fans. Let
$\mathbf{\Sigma} = (N, \Sigma, \beta)$ be a stacky fan. A \term{subfan}
$\mathbf{\Sigma'} \subset \mathbf{\Sigma}$ is a triple $\mathbf{\Sigma'}
= (N, \Sigma', \beta')$ where $\Sigma' \subset \Sigma$ is a
subset, which is a fan in its own right, and $\beta'$ is the restriction of
$\beta$ to $\ZZ^{\Sigma'(1)}$. The canonical map $\mathbf{\Sigma'} \to \mathbf{\Sigma}$,
which is the identity on $N$ corresponds to an open immersion
$X_{\mathbf{\Sigma'}} \to X_{\mathbf{\Sigma}}$. We say that
$X_{\mathbf{\Sigma'}}$ is a \term{toric open substack} of $X_{\mathbf{\Sigma}}$.
Of particular importance, are the toric substacks corresponding to stacky fans
generated by a single cone $\sigma \in \mathbf{\Sigma}$. We denote the
corresponding substack, which is a \term{basic} toric stack, by $U_\sigma$.

The coarse space of a toric stack $X_\mathbf{\Sigma}$ coincides with
the toric variety $X_\Sigma$ associated to the fan. The forgetful
functor from the category of stacky fans to the category of usual fans,
which simply forget the morphism $\beta$, commutes with the coarse space
functor.

\subsection{Toric stacky blow-ups}
For smooth toric varieties, blow-ups at orbit closures correspond to star subdivisions.
This generalises to toric orbifolds. We define what is meant by the star subdivision of
a stacky fan. This is the same definition as made by Edidin in \cite{em2012}.

\begin{definition}
Let $\mathbf{\Sigma} = (N, \Sigma, \beta)$ be a stacky fan. Let $\sigma$ be a
cone in $\Sigma$ and let $v = \sum_{\rho \in \sigma(1)} \beta(\rho)$.
Denote the ray generated by $v$ by $\rho_0$. We define the \term{star subdivision}
of the stacky fan $\mathbf{\Sigma}$ along $\sigma$ as
$\mathbf{\Sigma}^\ast(\sigma) = (N, \Sigma^\ast(v), \beta')$.
Here $\Sigma^\ast(v)$ denotes the subdivision of the fan $\Sigma$ obtained by
adding the ray $\rho_0$ and subdividing each cone containing it, as described in
\cite[§11.1]{cls2011}. The function $\beta'$ is the extension of $\beta$ to
$\ZZ^{\Sigma^\ast(v)}$ taking the ray $\rho_0$ to $v$. There is a canonical map
$\mathbf{\Sigma}^\ast(\sigma) \to \mathbf{\Sigma}$ which is the identity on $N$.
The ray $\rho_0$ is called the \term{exceptional ray} of the star subdivision.
\end{definition}

If $X_\mathbf{\Sigma}$ is the toric orbifold corresponding to the stacky fan
$\mathbf{\Sigma}$, and $\sigma$ is a cone in $\mathbf{\Sigma}$, then
the map $X_\mathbf{\Sigma^\ast(\sigma)} \to X_\mathbf{\Sigma}$ corresponding
to the star subdivision is the blow up of $X_\mathbf{\Sigma}$ with
centre $\VV(\sigma)$. The divisor $D_{\rho_0}$ on $X_\mathbf{\Sigma^\ast(\sigma)}$
corresponding to the exceptional ray $\rho_0$ is the exceptional divisor of the blow-up.

\begin{definition}
Let $\mathbf{\Sigma} = (N, \Sigma, \beta)$ be a stacky fan and
$\boldsymbol\rho = \{\rho_1, \ldots, \rho_r\} \subset \Sigma(1)$
a set of rays. For each $\rho_i \in \boldsymbol\rho$, we associate a
weight $d_i$, which is a positive integer. Denote the function
taking each ray to its weight by $\boldsymbol{d}$. Consider the group
homomorphism $\beta'\colon \ZZ^{\Sigma(1)} \to N$ defined by 
$$
\beta'(\rho) =
\left\{
\begin{array}{ll}
\boldsymbol{d}(\rho)\beta(\rho) & \text{if } \rho \in \boldsymbol\rho \\
\beta(\rho) & \text{otherwise.}
\end{array}
\right.
$$
We denote the stacky fan given by the triple $(N, \Sigma, \beta')$ by 
$\mathbf{\Sigma}_{\boldsymbol{d}^{-1}\boldsymbol\rho}$. The natural morphism
$\mathbf{\Sigma}_{\boldsymbol{d}^{-1}\boldsymbol\rho} \to \mathbf{\Sigma}$ of
stacky fans, which is the identity map on the underlying group $N$, is called
the \term{root construction} of $\mathbf{\Sigma}$ with respect to the rays in
$\boldsymbol\rho$ with weights $\boldsymbol{d}$.
\end{definition}

The terminology in the definition above is, of course, motivated by its
relation to the root stack of the corresponding toric stacks. Using the
same notation as in the definition above, we let $\pi\colon X' \to X$ be the
morphism of toric orbifolds associated to the root fan
$\mathbf{\Sigma}_{\boldsymbol{d}^{-1}\boldsymbol\rho} \to \mathbf{\Sigma}$.
On both $X$ and $X'$ we have toric divisors corresponding to the
rays $\rho_1, \ldots, \rho_r$. Denote the sets of such divisors by
$D = \{D_1, \ldots, D_r\}$ and $D' = \{D'_1, \ldots, D'_r\}$ respectively.
Then each divisor $D'_i$ is a $d_i$-th root of $\pi^\ast D_i$, and this
structure identifies $X' \to X$ with the root stack $X_{\boldsymbol{d}^{-1}D}
\to X$, where we consider $\boldsymbol{d}$ a function on $D$ in the obvious way.

In terms of homogeneous coordinates, the root stack of a basic toric stack
has the following description. Let $X$ be a basic toric stack with
homogeneous coordinates $(\sheaf{O}_S[x_1, \ldots, x_r], A, \boldsymbol{a})$.
Assume that $D$ is a set of toric divisors corresponding to the coordinates
$x_1, \ldots, x_s$ for some $s \leq r$.
Denote the generators of the group $\ZZ^s$ by $e_1, \ldots, e_s$ and define the 
group 
$$
A_{\boldsymbol{d}^{-1}\boldsymbol{a}} = A\oplus\ZZ^s/\langle d_1e_1 - a_1, \ldots, d_se_s - a_s \rangle,
$$
which we think of as the group obtained from $A$ by formally adjoining the
roots $e_i = a_i/d_i$. Also let
$\boldsymbol{a}' = (e_1, \ldots, e_s, a_{s+1}, \ldots, a_r)$. Then
the homogeneous coordinates of $X_{\boldsymbol{d}^{-1}D}$ is given by
$$
\left(\sheaf{O}_S[x_1^{1/d_1}, \ldots, x_s^{1/d_s}, x_{s+1} \ldots, x_r],
A_{\bs{d}^{-1}\bs{a}}, \bs{a}'\right).
$$
and the map $X_{\bs{d}^{-1}D} \to X$ corresponds to the map
of graded rings taking $x_i$ to $x_i$.

\subsection{Multiplicities and smoothness}
The toric destackification algorithm, which is described in the next section,
is based on the well-known toric desingularisation algorithm described
in for instance \cite[Sec.\ 11]{cls2011}. In particular, the \term{multiplicity}
of a cone plays an important role. Here we will briefly recall the main properties
of multiplicities. We will also introduce the related concept of
\term{independency} of toric divisors.

As usual, we let $\mathbf{\Sigma} = (N, \Sigma, \beta)$ be a stacky fan and
$\sigma \in \Sigma$ a cone. Let $\rho_1, \ldots, \rho_r$ be the rays in $\sigma(1)$,
and let $u_i$ be the non-zero lattice point on the ray $\rho_i$ which is closest to the origin.
We associate the parallelotope
$$
P_\sigma = \left\{\sum_{i = 1}^r \lambda_i u_i \mid 0 \leq \lambda_i < 1
\right\},
$$
to the cone $\sigma$. Then the number of lattice points in $P_\sigma$ is called
the \term{multiplicity} of $\sigma$ and is denoted by $\mult \sigma$. The multiplicity
satisfies the basic property $\mult \tau | \mult \sigma$ if $\tau \preceq \sigma$.
It should be noted that the stacky structure $\beta$ plays no part in the definition
of multiplicity. In particular, the multiplicity of a cone is preserved by the root
construction. The \term{multiplicity} $\mult \xi$ at a point $\xi \in X_{\mathbf{\Sigma}}$
in the toric orbifold $X$ is the multiplicity of the cone spanned by the rays corresponding
to the toric divisors passing through $\xi$.

We also describe the multiplicity for a basic toric stack $X$ with
homogeneous coordinates $(\sheaf{O}_S[x_1, \ldots, x_r], A, \bs{a})$.
Let $A_\div = \langle a_1, \ldots, a_r\rangle$, and define the quotient
group $A_i = A_\div/\langle a_1, \ldots, \widehat{a}_i, \ldots, a_r\rangle$ for
each element of $\bs{a}$. Then we have a natural exact sequence
$$
0 \to K \to A_\div \to A_1\times \cdots \times A_r \to 0.
$$
The \term{multiplicity} at the intersection of the toric divisors is the order
of $K$. It is straightforward to verify that this definition coincides with the
previous in the case when $X$ is a toric orbifold. Taking the
cartesian product of $X$ with a torus does not affect the multiplicity.

From the above description, we see that the multiplicity measures how far
$A_\div$ is from being a product of the quotients $A_i$. Another way to
measure this condition is given by \term{independency} of the toric divisors.

\begin{definition}
\label{def-independency-basic}
Let $(\sheaf{O}_S[x_1, \ldots, x_r], A, \boldsymbol{a})$ be the homogeneous
coordinates of a basic toric stack $X$. A toric divisor $D_i = V(x_i)$ is said
to be \term{independent} at the origin of $X$ if
$A_\div = \langle a_1, \ldots, \widehat{a}_i, \ldots a_r\rangle \oplus \langle
a_i\rangle$.
\end{definition}

We also have a corresponding combinatorial concept of independency.
\begin{definition}
\label{def-independency-combinatorial}
Let $\mathbf{\Sigma} = (N, \Sigma, \beta)$ be a stacky fan, $\sigma \in \Sigma$
a cone, and $\rho \in \sigma(1)$ a ray. We say that $\rho$ is \term{independent}
at $\sigma$ if $\mult \tau = \mult \sigma$ where $\tau$ is the face of $\sigma$
spanned by the rays $\sigma(1)\setminus \rho$.
\end{definition}
The definition is motivated by the fact that $\rho$ is independent at $\sigma$
if and only if $D_\rho$ is independent at the origin of $U_\sigma$.

\section{Toric destackification}
\label{sec-toric-destack}
Destackification of a toric orbifold may be performed by an algorithm which is
almost identical to the algorithm for resolving singularities of a simplicial
toric variety using sequences of star subdivisions, as described for instance in
\cite[§11]{cls2011}. At each step, we choose a cone of maximal multiplicity
and subdivide the cone at an appropriate ray in the interior of the cone.
This can be accomplished with stacky modifications by first taking roots of
the extremal rays and then using the stacky star subdivision of the cone itself.

The main problem with this approach is that functoriality with respect to toric open
immersion is not achieved. Taking a root modifies the associated toric orbifold
along the whole divisor. Thus a destackification algorithm can never be functorial
with respect to open immersions in a step by step fashion, if we take roots of divisors.

On the other hand, it is in general not possible to destackify by just using
stacky star subdivisions, as shown by the following example
(cf.~\cite[2.29.2]{kollar2007}).
\begin{example}
\label{example-toric-need-roots}
Let $X$ be a basic toric orbifold over a field $k$, with homogeneous coordinate ring
$(k[x_1, x_2], \ZZ/5\ZZ, (a, b))$. Blowing up at the origin gives two charts, which
are themselves basic toric stacks of the same form, but with weights $(a, b - a)$
and $(a - b, b)$ respectively. If we start with weight vector $(1, 3)$, one of the
charts have weight vector $(1, 2)$. But this basic toric stack is isomorphic to the
original one, since it can also be obtained by multiplying with 3 and permuting
the elements. Thus no improvement towards destackification has been achieved.
\end{example}

Our solution to the problem is similar to the one used in the classical strong
desingularisation algorithms. We relax the functoriality requirement and do not
demand the process to be functorial with respect to open immersions for each step.
This requires us to somehow keep track of the history of the destackification
process. We do this by adding additional structure to our toric orbifolds.

First of all, we will assume that the rays of the stacky fan are ordered. Note that
the ordering of the rays also induces an ordering on the cones, which is
induced by the lexicographic ordering of the power set of the set of rays.
This assures that the pair $(X, \bs{E})$, where $X$ is the toric stack and $\bs{E}$ is the
set of toric divisors, is a standard pair. Secondly, we use the the
concept of distinguished divisors introduced in Definition~\ref{def-distinguished}. The
concept translates to the combinatorial language of stacky fans in an obvious manner.

Roughly, destackification is achieved as follows. We blow up the most singular part
and mark the exceptional divisor as \term{distinguished}. Then we make the distinguished
divisors \term{independent} by using a sequence of \term{admissible} stacky blow-ups.
This is described in Algorithm~\ref{alg-part-toric}. In particular, only the locus
lying over the original problematic locus will be modified. This ensures
that destackifying the whole toric stack is compatible with destackifying each toric
open substack separately and then gluing together. The over-all process is described
in more detail in Algorithm~\ref{alg-toric-destack}.

\begin{algorithm}[Partial Toric Destackification]
\label{alg-part-toric}
The input of the algorithm is a stacky fan $\mathbf{\Sigma}_0$ with distinguished structure.
The output is a sequence
$$
\mathbf{\Sigma}_n \to \cdots \to \mathbf{\Sigma}_0
$$
of admissible stacky modifications, with the property that all distinguished rays of
$\mathbf{\Sigma}_n$ are independent. The construction is functorial with respect
to isomorphisms of stacky fans preserving the distinguished structure.
We use the notation $\mathbf{\Sigma}_i = (N, \Sigma_i, \beta_i)$ in the
description of the algorithm.
\begin{itemize}
\itemsep0em
\item[\bf A0.][Initialise] Set $i = 0$.
  
\item[\bf A1.][Check if finished]
Let $\mathcal{S}$ be the set of cones $\sigma \in \mathbf{\Sigma}_i$
such that $\sigma(1)$ contains a distinguished divisor and such that
the relative interior of the parallelotope $P_\sigma$ contains a lattice
point. If $\mathcal{S}$ is empty, the algorithm terminates.
\item[\bf A2.][Choose a formal sum of rays]
Order the cones in $\mathcal{S}$ first by the number of non-distinguished
extremal rays and then by the multiplicity. Let $\mathcal{S}_{\max}$ be
the subset of cones in $\mathcal{S}$ which are maximal with respect to this
ordering. Consider the set $\mathcal{P}$ of formal sums $\psi$ of rays
such that the ray $\beta_i(\psi)$ passes through a lattice point in $P_\sigma$
for some $\sigma \in \mathcal{S}_{\max}$. This set is non-empty by
construction. Let $\psi_i$ be the smallest element of $\mathcal{P}$ with
respect to the lexicographic ordering.

\item[\bf A3.][Root distinguished rays]
Assume that $\psi_i = d_1\rho_1 + \cdots + d_s\rho_s + c_1\delta_1 + \cdots + c_r\delta_r$, 
with $\rho_j$ and $\delta_j$ being distinct non-distinguished and distinguished rays
respectively. Let $\mathbf{\Sigma}_{i + 1} \to \mathbf{\Sigma}_i$ be the root construction
$\left(\mathbf{\Sigma}_i\right)_{c_1^{-1}\delta_1, \ldots, c_r^{-1}\delta_r} \to \mathbf{\Sigma}_i$,
and $\psi_{i + 1} = d_1\rho_1 + \cdots + d_s\rho_s + \delta_1 + \cdots + \delta_r$.
Increment $i$ by one. Note that after this step all distinguished rays in the
support of $\psi_{i}$ have coefficient one. Also, the transformation rule asserts
that $\beta_i(\psi_i) = \beta_{i-1}(\psi_{i-1})$.

\item[\bf A4.][Perform a stacky star subdivision]
Let $\sigma_i$ be the cone generated by the support of $\psi_i$.
Let $\mathbf{\Sigma}_{i+1} \to \mathbf{\Sigma}_i$ be the stacky star subdivision
$\mathbf{\Sigma}_{i}(\sigma_i)\to\mathbf{\Sigma}_i$
and denote the exceptional ray by $\varepsilon_{i + 1}$. Furthermore let
$
\psi_{i + 1} = \psi_i - \sum_{\rho \in \sigma_i(1)} \rho + \varepsilon_{i + 1},
$
and then increment $i$ by~1. Note that after this step the support of $\psi_{i}$
contains just one distinguished ray $\varepsilon_i$, which occurs
with coefficient one. Also, the transformation rule asserts that
$\beta_i(\psi_i) = \beta_{i-1}(\psi_{i-1})$.
\item[\bf A5.][Iterate inner loop]
While the support of $\psi_i$ contains more than one ray, repeat from Step~A4.
\item[\bf A6.][Iterate main loop]
Repeat from Step~A1.
\end{itemize}
\end{algorithm}
\begin{proof}[Proof of correctness of Algorithm~A]
Functoriality is clear, since all choices in the algorithm only depend on
properties preserved by isomorphisms.

If $\sigma$ is a cone containing a distinguished, non-independent ray $\delta$,
then there is a face $\sigma'$ of $\sigma$ containing $\delta$ with
$P_{\sigma'}$ containing a lattice point in its relative interior.
Hence the algorithm does not halt prematurely.

It remains to prove that the algorithm halts. For notational convenience,
we assume, without loss of generality, that $i = 0$ at the beginning of an
iteration of the main loop and $i = n$ when the iteration ends.

Denote the cone generated by the support of $\psi_0$ by $\sigma_0$, and let
$\tau_0$ be any cone in $\bs{\Sigma}_0$ of maximal dimension containing
$\sigma_0$. Using the notation in Step~A3, we have
$$
\tau_0 = \cone(\rho_1, \ldots, \rho_s, \delta_1, \ldots, \delta_r, \nu_1,
\ldots, \nu_t),
$$
for some rays $\nu_1, \ldots, \nu_t$. By maximality of $\sigma_0$ with respect
to the ordering defined in Step~A2, we have $\mult\tau_0 = \mult \sigma_0$.
Define $\tau_{i + 1}$ recursively as any choice of cone of maximal
dimension in the subdivision of $\tau_i$ such that $\tau_{i+1}$ has the same
number of non-distinguished rays as $\tau_i$. For $i \geq 2$, we have
$$
\tau_i = \cone(\rho_1, \ldots, \rho_s, \delta_1, \ldots, \widehat{\delta}_k, \ldots,
\delta_r, \nu_1, \ldots, \nu_t, \varepsilon_i),
$$
where $\widehat{\delta}_k$ indicates that the ray $\delta_k$ should be omitted
from the list for some $k$ with $1 \leq k \leq r$.

The transformation rule for the elements $\psi_i$ asserts that
$\beta_{i+1}(\psi_{i + 1}) = \beta_i(\psi_i)$ throughout a whole
iteration of the main loop. In particular, we have $\beta_n(\psi_n) =
\beta_n(\varepsilon_n) = \beta_0(\psi_0)$. But the ray through
$\beta_0(\psi_0)$ passes through
a lattice point in $P_{\sigma_0} \subset P_{\tau_0}$ by choice of $\psi_0$.
It follows that the multiplicity of $\tau_n$ is strictly smaller than
$\mult \tau_0$. Since any cone produced in the iteration of the main loop is
a face of $\tau_n$ for some choice of sequence $\tau_0, \ldots, \tau_n$,
it follows that all new cones are smaller than $\sigma_0$ with respect to the
ordering defined in Step~A2. Since $\sigma_0$ has been removed, this process
cannot continue indefinitely, and the algorithm eventually stops.
\end{proof}

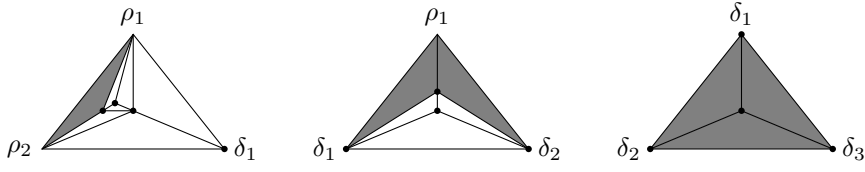
\begin{figure}
\begin{tikzpicture}[scale=0.8]
\begin{scope}
\coordinate [label=above:$\rho_1$] (a) at (1.5, 1.9);
\coordinate [label=left:$\rho_2$] (b) at (0, 0);
\coordinate [label=right:$\delta_1$] (c) at (3, 0);

\coordinate (d) at (barycentric cs:a=1,b=1,c=1);
\coordinate (e) at (barycentric cs:a=2,b=2,c=1);
\coordinate (f) at (barycentric cs:a=2,b=3,c=1);

\fill [gray] (a)--(b)--(f)--cycle;

\draw (a)--(b)--(c)--cycle;
\draw (d)--(a) (d)--(b) (d)--(c);
\draw (e)--(a) (e)--(b) (e)--(d);
\draw (f)--(a) (f)--(d);

\fill (c) circle (1.5 pt);
\fill (d) circle (1.5 pt);
\fill (e) circle (1.5 pt);
\fill (f) circle (1.5 pt);
\end{scope}

\begin{scope}[shift={(5, 0)}]
\coordinate [label=above:$\rho_1$] (a) at (1.5, 1.9);
\coordinate [label=left:$\delta_1$] (b) at (0, 0);
\coordinate [label=right:$\delta_2$] (c) at (3, 0);

\coordinate (d) at (barycentric cs:a=1,b=1,c=1);
\coordinate (e) at (barycentric cs:a=2,b=1,c=1);

\fill [gray] (a)--(b)--(e)--cycle;
\fill [gray] (a)--(c)--(e)--cycle;

\draw (a)--(b)--(c)--cycle;
\draw (d)--(a) (d)--(b) (d)--(c);
\draw (e)--(b) (e)--(c);

\fill (b) circle (1.5 pt);
\fill (c) circle (1.5 pt);
\fill (d) circle (1.5 pt);
\fill (e) circle (1.5 pt);
\end{scope}
\begin{scope}[shift={(10, 0)}]
\coordinate [label=above:$\delta_1$] (a) at (1.5, 1.9);
\coordinate [label=left:$\delta_2$] (b) at (0, 0);
\coordinate [label=right:$\delta_3$] (c) at (3, 0);

\coordinate (d) at (barycentric cs:a=1,b=1,c=1);

\fill [gray] (a)--(b)--(d)--cycle;
\fill [gray] (a)--(c)--(d)--cycle;
\fill [gray] (b)--(c)--(d)--cycle;

\draw (a)--(b)--(c)--cycle;
\draw (d)--(a) (d)--(b) (d)--(c);

\fill (a) circle (1.5 pt);
\fill (b) circle (1.5 pt);
\fill (c) circle (1.5 pt);
\fill (d) circle (1.5 pt);
\end{scope}
\end{tikzpicture}
\caption{\label{fig_subdivision}
Examples of the subdivision procedure in Step~4 of Algorithm~A.}
\end{figure}

The inner workings of Algorithm~A are best illustrated with examples. 
Figure~\ref{fig_subdivision} illustrates the subdivision process in the steps A4
and~A5 in three different cases. In each of the examples, we start with a fan generated by a single
cone, and describe the subdivision obtained during a single iteration of the main loop. To make
the example easier to draw, we just draw the intersection of the cone with the plane
through the marked lattice points on the extremal rays. The rays are the corners of the
triangles, and the distinguished rays are marked by black dots.
The grey triangles show the cones where the multiplicities have dropped at the end of
the iteration. The white triangles may have higher multiplicity, but they have
fewer of non-distinguished rays. In the first example, we start with
$\psi_0 = 2\rho_1 + 3\rho_2 + \delta_1$. In the second example, we start with
$\psi_0 = 2\rho_1 + \delta_1 + \delta_2$. In the final example,
we start with $\psi_= \delta_1 + \delta_2 + \delta_3$. In the final case,
where all rays are distinguished, the algorithm degenerates to the naïve
algorithm mentioned in the beginning of the section.

By invoking Algorithm~\ref{alg-part-toric} repeatedly, we get a functorial toric destackification
algorithm. The process is explicitly in Algorithm~B, but we skip the easy proof, since this is
a special case of the much more general Algorithm~\ref{alg-destack}.

\begin{algorithm}[Functorial Toric Destackification]
\label{alg-toric-destack}
The input of the algorithm is a stacky fan $\mathbf{\Sigma}_0$ with ordered structure.
The output is a sequence
$$
\mathbf{\Sigma}_n \to \cdots \to \mathbf{\Sigma}_0
$$
of stacky modifications such that all rays in $\mathbf{\Sigma}_n(1)$ are independent. That is,
all cones in $\mathbf{\Sigma}_n(1)$ are smooth. The construction is functorial with respect to
isomorphisms and taking subfans of stacky fans with ordered structure.
\begin{itemize}
\itemsep0em  
\item[\bf B0.][Initialise] Set $i = 0$.
\item[\bf B1.][Choose a cone] Consider the set $\mathcal{S}$ of cones $\sigma$ in
$\mathbf{\Sigma}_i$ with the property that none of the rays in $\sigma(1)$ are
independent in $\sigma$. If this set is empty, then all rays in $\mathbf{\Sigma}_i$
are independent and the algorithm terminates. Choose a cone $\sigma \in \mathcal{S}$
of maximal dimension. If several such cones exist, choose the largest one with
respect to the natural ordering on the cones in $\mathbf{\Sigma}_i$.
\item[\bf B2.][Create distinguished ray]
Let $\mathbf{\Sigma}_{i+1} \to \mathbf{\Sigma}_i$ be the star subdivision of
$\mathbf{\Sigma}_i$ in $\sigma$. Increment $i$ by one.
\item[\bf B3.][Resolve the cone]
Give $\mathbf{\Sigma}_i$ a distinguished structure, by letting the exceptional
ray from the subdivision be the only distinguished ray. Invoke Algorithm~A
and append the output to the sequence. Increment $i$ by the length of this output.
\item[\bf B4.][Iterate]
Forget the distinguished structure, and iterate from Step~B1.
\end{itemize}
\end{algorithm}

\section{Local homogeneous coordinates}
\label{sec-local-homogeneous}
In Section~\ref{sec-toric-stacks}, we introduced \term{basic toric stacks}. Here
we will show that each smooth tame stack with diagonalisable stabilisers is
étale locally of this form. This will allow us to use {\em local} homogeneous
coordinates even for non-toric stacks, which in turn will allow us to generalise
the toric destackification algorithm. We start by making a precise definition
of what we mean by a stack being locally toric.

\begin{definition}
Let $X$ be a smooth algebraic stack over a scheme $S$, and let $\xi \in X$ be a point.
By a \term{toric chart} of $X$ over $S$ at $\xi$, we mean a diagram
$$
\xymatrix{
& X' \ar[dl]_f \ar[dr]^g &\\
X & & X''\\
}
$$
of algebraic stacks over $S$, together with a point $\xi' \in X'$. The data are
required to satisfy the following properties:
\begin{enumerate}
\item
The stack $X''$ is a basic toric stack over $S$.
\item
The point $\xi'$ maps to $\xi$ in $X$ and to a point $\xi''\in X''$ lying in the
intersection of the prime toric divisors of $X''$.
\item
The maps $f$ and $g$ are étale and stabiliser preserving.
\end{enumerate}
A homogeneous coordinate ring of $X''$ is called a \term{local homogeneous coordinate
ring} at $\xi$. Assume that $E$ is a simple normal crossings divisor on $X$ and
$Z$ a closed substack of $X$ having simple normal crossings only with $E$. Then we say
that $E$ and $Z$ are \term{compatible} with the toric chart if the pull back of $E$
to $X'$ coincides with the pull-back of a toric divisor on $X''$, and the pull-back of $Z$
to $X'$ coincides with the pull-back of an intersection of prime toric divisors on $X''$.
\end{definition}

To prove that a smooth tame stack with diagonalisable stabilisers has a toric
chart at every point, we need a version of the structure theorem for tame
algebraic stacks which takes smoothness into account. We give such a
theorem in Appendix~\ref{appendix-tame}. We will also need the following
lemma.

\begin{lemma}
\label{lemma-refine-invariant}
Let $X = \Spec A$ be an affine scheme over an affine base scheme $S$, and let
$G$ be a finite, linearly reductive, locally free group scheme over $S$ acting
on $X$. Let $\xi \in X$ be a point, and let $D(f)$ be a
distinguished open subscheme of $X$ containing the orbit of $\xi$. Then there is
a refinement $\xi \in D(g) \subset D(f)$ such that $g$ is an invariant section
which is a multiple of $f$.
\end{lemma}
\begin{proof}
Let $\pi\colon X \to X/G = \Spec A_0$ be the coarse quotient, where
$A_0 \subset A$ is the ring of invariant sections. The map $\pi$
is integral and therefore closed. The set $\pi(V(f))$ does not contain
$\pi(\xi)$, by the assumption that the orbit of $\xi$ is contained in
$D(f)$. Let $D(h)$, with $h \in A_0$ be a distinguished open
neighbourhood of $\pi(\xi)$ in the complement of $\pi(V(f))$ in
$X/G$. This pulls back to an open subset, also denoted by $D(h)$,
satisfying $\xi \in D(h) \subset D(f)$. The condition $D(h) \subset D(f)$
implies that $\rad(h) \subset \rad(f)$. Hence, there is a power $g = h^n$
of $h$ which is a multiple of $f$.
\end{proof}

Now we are ready for the main theorem of this section.

\begin{prop}
Let $X$ be an algebraic stack with finite inertia and diagonalisable geometric stabilisers. Assume that
$X$ is smooth and quasi-separated over a scheme $S$. Then $X$ admits toric charts over $S$ at each of
its points. Furthermore, if $E$ is a simple normal crossings divisor on $X$, and $Z$ is a closed substack
of $X$ having simple normal crossings with $E$, then the toric charts may be chosen such that they
are compatible with $E$ and $Z$.
\end{prop}
\begin{proof}
The question may be verified stabiliser preserving étale locally on $X$, so
by Propositions~\ref{prop-tame-fixed} and~\ref{prop-fixed-smooth}, we may
assume that $X$ is of the form $[U/\Delta]$, where $U$ is an affine scheme which
is smooth over $S$ and $\Delta$ is a diagonalisable group acting on $U$.
Furthermore, we may assume that $\xi$ lifts to a point $\xi' \in U$ which is
fixed under the $\Delta$-action.

The $\Delta$-action corresponds to a grading on $\sheaf{O}_U$ by
the Cartier dual $\Delta^\vee$, which is a finite abelian group. Choose
homogeneous global sections $f_1, \ldots, f_n$ of $\sheaf{O}_U$ such that
the differentials $df_1, \ldots, df_n$ form
a basis of $\Omega_{U/S}\otimes_{\sheaf{O}_U}\kappa(\xi')$. Consider the map
$\sheaf{O}_S[x_1, \ldots, x_n] \to \sheaf{O}_U$ taking $x_i$ to $f_i$. We give
the polynomial ring a $\Delta^\vee$-graded structure, by letting $x_i$ have
the same degree as $f_i$. This gives an equivariant map $\widehat{g}\colon U \to
\AA^n_S$ over $S$. By construction, the canonical map
$$
\Omega_{\AA^n_S/S}\otimes_{\sheaf{O}_S[x_1, \ldots, x_n]}\kappa(\xi') \to
\Omega_{U/S}\otimes_{\sheaf{O}_U}\kappa(\xi')
$$
is an isomorphism. Since $U$ is smooth over $S$, it follows that $\widehat{g}$
is étale at $\xi'$ by \cite[17.11.2]{egaIV4}. Denote the corresponding map
$[U/\Delta] \to [\AA^n_S/\Delta]$ of stacks by $g$.
The map $g$ is representable, so the stabiliser of $\xi$ injects into the
stabiliser of $g(\xi)$. Since $\xi'$ is fixed by the action of $\Delta$, the
stabiliser at $\xi$ is $\Delta$, so the map of stabiliser must be an isomorphism.
Since the locus where $f$ is étale and stabiliser preserving is open
\cite[Prop.~6.5]{rydh2013}, we just as well assume that $[U/\Delta] \to
[\AA^n_S/\Delta]$ is étale and stabiliser preserving, after shrinking $U$ invariantly
by using Lemma~\ref{lemma-refine-invariant} if necessary.
Finally, we simply remove the prime toric divisors from $[\AA^n_S/\Delta]$ which
do not contain $g(\xi)$.

Now we turn to the statement about the simple normal crossings divisors. Let
$E_1, \ldots, E_r$ be the components of $E$ passing through $\xi$. They
correspond to locally principal homogeneous ideals $I_i$ in $\sheaf{O}_U$.
Also denote the homogeneous ideal corresponding to $Z$ by $I$. Next we
choose our sections $f_1, \ldots, f_n$ one by one in a way such that the
differentials $df_i$ remain linearly independent in
$\Omega_{U/S}\otimes_{\sheaf{O}_U}\kappa(\xi')$. First we pick homogeneous $f_i$
from $I_i$, for $1 \leq i \leq r$. Then we pick homogeneous $f_{r+1}, \ldots,
f_s$ from $I$ with $s$ as large as possible. Finally, we pick the remaining
homogeneous sections from $\sheaf{O}_U$. By the normal crossings assumption, we
get compatibility in a neighbourhood of $\xi'$, which we may assume is $\Delta$-invariant by
Lemma~\ref{lemma-refine-invariant}.
\end{proof}

\section{The conormal representation}
\label{sec-conormal}
In the destackification algorithms, several different invariants will be used
in order to determine appropriate loci to blow up. In this section, we will
develop an abstract framework in which common properties of these invariants
will be studied.

We fix some notation, which will be used throughout the section. Let
$(X, \bs{E})$ be a standard pair over a scheme $S$, with $X$ having diagonalisable
stabilisers. The stabiliser at a geometric point $\xi\colon\Spec \bar{k} \to X$
will be denoted $\Delta_{\xi}$ and its group of characters by $A(\xi)$. The set
of components of $\bs{E}$ passing through $\xi$ will be denoted by $\bs{E}(\xi)$.

Let $X_{\bar{k}}$ be the pull-back of $X$ along the composition
$\Spec \bar{k} \to X \to S$. Then the morphism $\xi$ factors as
$$
\Spec\bar{k} \to \BB\Delta_{\xi} \hookrightarrow X_{\bar{k}} \to X.
$$
The map $\Spec\bar{k} \to X_{\bar{k}}$ is a section of the natural projection.
By Lemma~\ref{rational-point-closed} this implies that the canonical
monomorphism $\BB\Delta_{\xi} \hookrightarrow X_{\bar{k}}$ is a closed
immersion. Recall that the category of coherent sheaves of
$\sheaf{O}_{\BB\Delta_{\xi}}$-modules is equivalent to the category of finite
dimensional $\Delta_{\xi}$-representations over $\bar{k}$.
We call the $\Delta_\xi$-representation $V(\xi)$ corresponding to
the conormal bundle $\sheaf{N}_{\BB\Delta_\xi/X_{\bar{k}}}$ the
\term{conormal representation} at $\xi$.

The presence of the ordered set of divisors $\bs{E}$ on $X$ gives the conormal
representation at each point $\xi \in X$ some extra structure. Some of the
components in the splitting of the conormal representation $V(\xi)$ into
one-dimensional representations will be marked by the components of
$\bs{E}(\xi)$ in a way made precise by the following proposition.

\begin{prop}
\label{prop-struct-conormal}
Let $(X, \bs{E})$ be standard pair over a scheme $S$, and assume that $X$
has diagonalisable stabilisers. Given a geometric point $\xi\colon\Spec \bar{k}
\to X$, we let $E^1, \ldots, E^r$ be the components of $\bs{E}(\xi)$.
Let $g_i\colon \BB\Delta_\xi \hookrightarrow E^i_{\bar{k}}$
denote the canonical morphism to the fibre of the component $E^i$. Then the
conormal representation $V$ at $\xi$ splits into a direct sum
$$
V = V_1 \oplus \cdots \oplus V_r \oplus V_\res,
$$
where each $V_i$ is one-dimensional and corresponds to the pull-back
$g^\ast_i\sheaf{N}_{E^i_{\bar{k}}/X_{\bar{k}}}$ of the conormal bundle
corresponding to the divisor $E^i$.
\end{prop}
\begin{proof}
By passing to the fibre, we may, without loss of generality, assume that $S = \Spec\bar{k}$.
Let $Z_0 = X$ and define $Z_i$ recursively by means of the cartesian diagrams
$$
\xymatrix{
Z_i\ar[r]\ar[d]_{h_i} & Z_{i-1}\ar[d]\\
E^i\ar[r] & X.
}
$$
Since we assume that the divisors $E^i$ intersect transversally, each $Z_i$ is smooth
and we have canonical isomorphisms $\sheaf{N}_{Z_i/Z_{i-1}} \simeq
h_i^\ast\sheaf{N}_{E^i/X}$ by \cite[Prop.\ 17.13.2]{egaIV4}. Now consider the increasing filtration
$$
\BB\Delta_\xi \hookrightarrow Z_r \hookrightarrow \cdots \hookrightarrow Z_0 = X
$$
of closed immersions between stacks which are smooth over $S$. We denote the
various compositions by $f_i\colon\BB\Delta_\xi \to Z_i$. By
\cite[Prop.\ 16.9.13]{egaIV4}, we have short exact sequences
$$
0 \to f_i^\ast\sheaf{N}_{Z_{i}/Z_{i-1}}
\to \sheaf{N}_{\BB\Delta_\xi/Z_{i-1}}
\to \sheaf{N}_{\BB\Delta_\xi/Z_i} \to 0
$$
for each $i \in \{1, \ldots, r\}$. Since the group $\Delta_\xi$ is linearly reductive,
these sequences split, and we get a decomposition
$$
\sheaf{N}_{\BB\Delta_\xi/X} =
f_1^\ast\sheaf{N}_{Z_1/Z_0} \oplus \cdots \oplus f_r^\ast\sheaf{N}_{Z_r/Z_{r-1}}
\oplus \sheaf{N}_{\BB\Delta_\xi/Z_r}.
$$
But the maps $g_i$ factors through $h_i$, which implies that we get canonical
isomorphisms $g_i^\ast\sheaf{N}_{E^i/X} \simeq f_i^\ast\sheaf{N}_{Z_i/Z_{i-1}}$. We therefore
get the desired decomposition by letting $V_\res$ be the representation corresponding
to $\sheaf{N}_{\BB\Delta_\xi/Z_r}$. Since the substacks $E^i$ are effective Cartier
divisors, the bundles $\sheaf{N}_{E^i/X}$ are locally free of rank one. This shows that
each $V_i$ is one-dimensional, which concludes the proof.
\end{proof}

Using the notation of Proposition~\ref{prop-struct-conormal}, we introduce some
terminology to describe the extra structure induced by the ordered set of divisors.
The subrepresentation $V_1 \oplus \cdots \oplus V_r$ is called the
\term{divisorial part} of the conormal representation, and
$V_\res$ is called the \term{residual part}. The representation
$V_\res$ can be further split up in a sum $V' \oplus V''$, where
$V'$ is a direct sum of one-dimensional non-trivial representations
and $V''$ is a direct sum of one-dimensional trivial representations.
We call $V''$ the \term{irrelevant part} and $V'$ the \term{relevant
residual part}. The \term{relevant part} is the sum of the relevant
residual part $V'$ and the divisorial part $V_1 \oplus \cdots \oplus V_r$.

For the purpose of constructing invariants, we are only interested in
conormal representations up to isomorphism. In addition, we do not want
our invariants to depend on the choice of geometric point representing $\xi$.
Hence it makes sense to pass to the representation ring of $\Delta_\xi$, or
equivalently, to the Grothendieck group $\KK_0(\coh{\BB\Delta_\xi})$.

Note that since $\Delta_\xi$ is assumed to be diagonalisable, the structure of
the group $\KK_0(\coh{\BB\Delta_\xi})$ is particularly simple.
Each $\Delta_\xi$-representation splits into one-dimensional
representations corresponding to characters of $\Delta_\xi$. Hence 
we have a canonical isomorphism $F(A(\xi)) \to \KK_0(\coh{\BB\Delta_\xi})$,
where $F(A(\xi))$ denotes the free group on the set $A(\xi)$
of characters for $\Delta(\xi)$. In the sequel we shall identify
these groups.

The additional structure given by $\bs{E}$, as described in
Proposition~\ref{prop-struct-conormal} can be modelled as a function
$\bs{E}(\xi) \to F(A(\xi))$ factoring through $A(\xi)$.

We formalise the situation as follows. Fix a finite, totally ordered set $C$,
and define the set $U(C)$ as the set of equivalence classes of quadruples
$$
(A, v \in F(A), C_0 \subseteq C, \mu\colon C_0 \to F(A))
$$
with $A$ being a finite abelian group. We require the quadruples to satisfy
the following properties.
\begin{enumerate}
\item[(i)] The function $\mu$ factors through $A$. We denote
the sum $\sum_{c\in C_0} \mu(c)$ by $v_\div$.
\item[(ii)] The element $v$ can be written as a sum $v = v_\div + v_\res$
where $v_\res$ has positive coefficients.
\end{enumerate}
The quadruple 
$$
(A', v' \in F(A'), C'_0 \subseteq C, \mu'\colon C'_0 \to F(A'))
$$
is equivalent to the quadruple above provided that $C'_0 = C_0$ and there
exists an isomorphism $A \to A'$ such that $v$ maps to $v'$ and $\mu'$
equals the composition of $\mu$ with the canonical morphism $F(A) \to F(A')$.
 
\begin{definition}
\label{def-conormal}
Let $C$ be a totally ordered set. The set $U(C)$ described above is called
the set of \term{universal conormal invariants}. The \term{universal conormal
invariant} for a standard pair $(X, \bs{E})/S$ with diagonalisable stabilisers
is the function
$
u_{(X, \bs{E})/S}\colon |X| \to U(\bs{E})
$
given by
$$
\xi \mapsto
\left(
A(\xi),
[\sheaf{N}_{\BB\Delta_\xi/X_{\bar{k}}}],
\bs{E}(\xi),
E^i \mapsto [g_i^\ast\sheaf{N}_{E^i_{\bar{k}}/X_{\bar{k}}}]
\right),
$$
using the notation from the statement of
Proposition~\ref{prop-struct-conormal}.
\end{definition}

Given a quadruple $(A, v, C_0, \mu)$, we apply the terms
\term{residual}, \term{divisorial}, \term{relevant}
and~\term{irrelevant} to different parts of the splitting
of $v$ in a similar way as we do for conormal representations.
In practice, it will be cumbersome to work with the
quadruples introduced above. Hence we will prefer to describe
the universal conormal invariants as representations with certain
marked subrepresentations in the sections to follow, but in this
section we shall mostly keep the more formal point of view.

The group $\KK_0(\coh{\BB\Delta_\xi})$ can also be identified with
$\KK_0(\perf{\BB\Delta_\xi})$, which is the Grothendieck group of the
triangulated category of perfect complexes over $\BB\Delta_\xi$. In
Appendix~\ref{appendix-cotangent-conormal}, we will see how
the class of the conormal representation in this group can be viewed as
the class of the derived pullback of the cotangent complex of $X$ over $S$.

Our next step is to introduce an ordering on $U(C)$ which will
respect the topology on $X$. We start by describing the universal
conormal invariants for a basic toric stack.

\begin{prop}
\label{prop-basic-conormal}
Let $S$ be a scheme and $(X, \bs{E})/S$ a basic toric stack with
homogeneous coordinate ring $(\sheaf{O}_S[x_1, \ldots, x_r], A, \bs{a})$.
Let $\xi\colon\Spec\bar{k} \to X$ be a geometric point, and let $J
\subset \{1, \ldots, r\}$ be the subset of indices corresponding to
divisors in $\bs{E}(\xi)$. Then we have a surjection $\varphi\colon A \to
A(\xi)$ to the character group of the stabiliser at $\xi$. The kernel of
$\varphi$ is the subgroup generated by elements $a_i$ such that $i \not\in J$.
The conormal representation at $\xi$ decomposes as
$$
V = V_1 \oplus \cdots \oplus V_r
$$
into one-dimensional subspaces. The subspace $V_i$ has degree $\varphi(a_i)$
and corresponds to the component $E^i$ precisely when $i \in J$. The residual
part of $V$ is the sum $V_\res = \bigoplus_{i \not\in J} V_i$, and is
irrelevant.
\end{prop}
\begin{proof}
Since the field $\bar{k}$ is algebraically closed, the map $\xi$
factors through the atlas $\Spec_{\sheaf{O}_S}\sheaf{O}_S[x_1, \ldots, x_r]$.
Let $\alpha_i$ be the image of $x_i$ through the corresponding map
$$
\Gamma(\sheaf{O}_S[x_1, \ldots, x_r]) \to \bar{k}.
$$
Then we have $\alpha_i = 0$ precisely when $\xi$ passes through $E^i$.

Consider the atlas $\widetilde{X}_{\bar{k}} = \Spec\bar{k}[x_1, \ldots, x_r]$
of $X_{\bar{k}}$. The closed immersion $\BB\Delta_\xi \hookrightarrow
X_{\bar{k}}$ corresponds to the slice of the action groupoid at the
closed subscheme $V(I) \subset \widetilde{X}_{\bar{k}}$ defined by the
ideal $I = (x_1 - \alpha_1, \ldots, x_r - \alpha_r)$. In other
words, we have the cartesian diagram
$$
\xymatrix{
\Delta_\xi \ar[r]^-{g'} \ar[d] &
\widetilde{X}_{\bar{k}}\times_{\bar{k}}\Delta \ar[d]^f\\
\Spec\bar{k} \ar[r]_-g &
\widetilde{X}_{\bar{k}}\times_{\bar{k}}\widetilde{X}_{\bar{k}}.
}
$$
The map $g$ corresponds to the $\bar{k}$-algebra map
$
\bar{k}[x_1, \ldots, x_r, y_1, \ldots y_r] \to \bar{k}
$
taking $x_i$ and $y_i$ to $\alpha_i$, and the map $f$ corresponds to the
$\bar{k}$-algebra map
$$
\bar{k}[x_1, \ldots, x_r, y_1, \ldots y_r] \to \bar{k}[x_1, \ldots, x_r][A]
$$
taking $x_i$ to $x_i$ and $y_i$ to $a_ix_i$. It follows that the map
$g'$ corresponds to
$$
\bar{k}[x_1, \ldots, x_r][A] \to \bar{k}[A]/(\alpha_i - a_i\alpha_i),
$$
taking $x_i$ to $\alpha_i$, where $i$ ranges from $1$ to $r$.
Since the relation $\alpha_i = a_i\alpha_i$ is trivial if $\alpha_i = 0$ and
equivalent to $a_i = 1$ otherwise, the right hand side is the group algebra
$\bar{k}[A(\xi)]$ in the statement of the proposition.
The conormal representation is the $\bar{k}$-vector space $I/I^2$, which has
the elements $e_i = (x_i - \alpha_i) + I^2$ for $i \in \{1, \ldots, r\}$ as
basis. Since $e_i$ has degree $\varphi(a_i)$ and corresponds to the divisor
$E^i$ precisely when $\alpha_i = 0$, the result follows.
\end{proof}

\begin{definition}
\label{def-conormal-ordering}
Let $\alpha$ and $\alpha'$ be elements of the set $U(C)$ and assume that
they are represented by the quadruples $(A, v, C_0, \mu)$ and
$(A', v', C_0', \mu')$ respectively. We introduce a relation $\geq$ on $U(C)$
by letting $\alpha \geq \alpha'$ provided that $C_0'$ is a subset of $C_0$
and there exists a surjective group homomorphism $\varphi\colon A \to A'$
satisfying the following properties.
\begin{enumerate}
\item[(i)] The natural map $F(A) \to F(A')$ takes $v$ to $v'$.
\item[(ii)] The kernel of $\varphi$ is generated by the elements $a \in A$ with positive coefficient
in $v$ satisfying $\varphi(a) = 0$.
\item[(iii)] For each $c \in C'_0$, the natural map $F(A) \to F(A')$ takes $\mu(c)$ to $\mu'(c)$.
\item[(iv)] For each $c \in C_0 \setminus C'_0$, the element $\mu(c)$, viewed as an element
of $A$ is in the kernel of $\varphi$.
\end{enumerate}
The relation $\geq$ is clearly well-defined and gives $U(C)$ the structure of a partially
ordered set.
\end{definition}

The association $C \mapsto U(C)$ extends in an obvious way to a
functor from the category $\mathcal{C}$ of totally ordered sets with injective
order preserving morphisms to the category of partially ordered sets. We are now
in the position to define what we mean with a conormal invariant.

\begin{definition}
A \term{conormal invariant} is a natural transformation $\iota\colon U \to W$
to some functor $W$ from the category $\mathcal{C}$ as defined above to the
category of partially ordered sets. Given a standard pair $(X, \bs{E})/S$ with
diagonalisable stabilisers, we define the \term{realisation}
$
\iota_{(X, \bs{E})/S}\colon |X| \to W(\bs{E})
$
of the conormal invariant $\iota$ as the function given by the composition
$\iota_{\bs{E}}\circ u_{(X, \bs{E})/S}$. We will frequently abuse the
terminology and use the same term for a conormal invariant as for its
realisation.
\end{definition}

In practice, the functor $W$ in the definition above will most often be
the constant functor which takes all objects to $\NN$. This will be true
for all but one of the conormal invariants introduced in the next section.

\begin{example}
\label{example-basic-conormal-invariants}
We give some simple examples of conormal invariants. We describe their realisations
for a standard pair $(X, \bs{E})/S$ with diagonalisable stabilisers. All the examples
but the last take their values among the natural numbers.
\begin{enumerate}
 \item The function taking each point $\xi \in X$ to the order of the stabiliser $\Delta_\xi$.
 \item The function taking each point $\xi \in X$ to the number $|\bs{E}(\xi)|$ of components of $\bs{E}$
 passing through $\xi$.
 \item The function taking each point $\xi \in X$ to the multiplicity, defined as in
 Section~\ref{sec-toric-stacks}, at the point.
 \item The function taking each point $\xi \in X$ to the element $\bs{E}(\xi)$ in the power set of $\bs{E}$
 ordered by inclusion.
\end{enumerate}
None of the conormal invariants in this example will actually be used in the destackification algorithms.
\end{example}

Let $(A, v, C_0, \kappa)$ be a quadruple representing an element in $U(C)$,
and denote the subgroup of $A$ generated by the support of $v$ by $A_\div$.
All the conormal invariants $\iota\colon\mathcal{U} \to W$ that we will use
in the destackification algorithms will satisfy both of the following two
properties.
\begin{enumerate}
\item[{\bf P1}]
We have $\iota(A, v, C_0, \kappa) = \iota(A, v + v', C_0, \kappa)$
if $v'$ is irrelevant.
\item[{\bf P2}]
We have $\iota(A, v, C_0, \kappa) = \iota(A_\div, v, C_0, \kappa)$,
where we consider $F(A_\div)$ as a subgroup of $F(A)$ and restrict
$\kappa$ accordingly.
\end{enumerate}
For basic toric stacks, the conditions can be interpreted
as follows. The first condition says that the invariant does not depend
on torus factors. The second condition says that the invariant does not depend
on the kernel of the group action defining the basic toric stack.
In general, the conditions lead to the functoriality properties described
in the following proposition.

\begin{prop}
\label{prop-conormal-functorial}
Let $\iota\colon U \to W$ be a conormal invariant, and
$(X, \bs{E})/S$ a standard pair with diagonalisable stabilisers. Consider the 2-commutative
diagrams
$$
\xymatrix{
Y \ar[r]^f \ar[dr] & X \ar[d]\\
& S \\
}
\qquad
\xymatrix{
X' \ar[r]^g \ar[d] & X \ar[d]\\
S' \ar[r] & S
}
$$
where $f$ is smooth and the square is 2-cartesian. Let $\bs{F}$ be the pullback
of $\bs{E}$ along $f$ and $\bs{E}'$ be the pullback of $\bs{E}$ along $g$.
Then $\iota_{(X', \bs{E}')/S'} = \iota_{(X, \bs{E})/S}\circ|g|$. Furthermore,
we have the equality $\iota_{(Y, \bs{F})/S} = \iota_{(X, \bs{E})/S}\circ|f|$
under either of the following circumstances:
\begin{enumerate}
\item The morphism $f$ is étale and stabiliser preserving.
\item The morphism $f$ is smooth and stabiliser preserving and $\iota$
satisfies property {\bf P1}.
\item The morphism $f$ is a gerbe and $\iota$
satisfies property {\bf P2}.
\end{enumerate}
\end{prop}
\begin{proof}
In the proof we will use the description of the conormal representation in terms of the
cotangent complex as described in Appendix~\ref{appendix-cotangent-conormal} freely.
Let $\xi'\colon\Spec\bar{k} \to X'$ be a geometric point. Since we have a canonical isomorphism
between $X'\times_S\Spec\bar{k}$ and $X\times_S\Spec\bar{k}$, functoriality with respect to
base change follows immediately.

We explore the other functoriality properties by examining the first diagram.
Let $\xi\colon\Spec\bar{k} \to Y$ be a geometric point. By the previous paragraph, we may,
without loss of generality, assume that $S = \Spec\bar{k}$. We get a pair of 2-commutative
diagrams
$$
\xymatrix{
\BB\Delta_\xi \ar[r]^b\ar[d]_a & Y \ar[d]^f \\
\BB\Delta_{f\circ\xi} \ar[r] & X \\
}
\qquad
\xymatrix{
\BB\Delta_\xi \ar[r]^{g_i}\ar[d]_a & F^i \ar[r]\ar[d]^{f_i} & Y \ar[d]^f\\
\BB\Delta_{f\circ\xi} \ar[r]_{h_i} & E^i \ar[r] & X \\
}
$$
where the rightmost square is 2-cartesian. By definition of realisation, the
values $\iota_{(Y, \bs{F})/S}(\xi)$ and $\iota_{(X, \bs{E})/S}(f\circ\xi)$ are 
$$
\iota(A(\xi), [\Nbl{\BB\Delta_\xi/Y}], \bs{F}(\xi) , F^i \mapsto g_i^\ast[\Nbl{F^i/Y}])
$$
and 
$$
\iota(A(f\circ\xi), [\Nbl{\BB\Delta_{f\circ\xi}/X}], \bs{E}(f\circ\xi) , E^i \mapsto h_i^\ast[\Nbl{E^i/X}])
$$
respectively. The set $\bs{F}(\xi)$ can clearly be identified with the corresponding set
$\bs{E}(f\circ\xi)$, via the natural bijection between $\bs{F}$ and $\bs{E}$.

We assume that the map $f$ is stabiliser preserving. Then we can identify $\Delta_\xi$ with $\Delta_{f\circ\xi}$
and assume that $a$ is the identity map. In particular, we have $A(\xi) =
A(f\circ\xi)$. Since conormal bundles commute with flat base change, we get
$$
h_i^\ast[\Nbl{E^i/X}] = g_i^\ast f_i^\ast[\Nbl{E^i/X}] = g_i^\ast[\Nbl{F^i/Y}]
$$
from the right diagram, so the divisorial part of the conormal representation is identical. By using the
distinguished triangle for composition on the left diagram, we get the identity
$$
[\Nbl{\BB\Delta_\xi/Y}] = [\Nbl{\BB\Delta_\xi/X}] + b^\ast[L_{Y/X}].
$$
The map $b$ factors through the fibre product $\widetilde{Y} = Y\times_X \BB\Delta_\xi$. Consider the
2-commutative diagram
$$
\xymatrix{
\widetilde{Y}_\coarse \ar[d] & \widetilde{Y} \ar[l]\ar[r]\ar[d]& Y \ar[d]\\
S & \BB\Delta_\xi \ar[l]\ar[r] & X. \\
}
$$
Since the map $\widetilde{Y} \to \BB\Delta_\xi$ is stabiliser preserving and $\BB\Delta_\xi$ is a gerbe,
the left square is 2-cartesian. Hence $b^\ast[L_{Y/X}]$ is equal to a pull-back
of the element $[L_{\widetilde{Y}_\coarse/S}]$. Since $\widetilde{Y}_\coarse$ is an algebraic space, the
representation corresponding to $b^\ast[L_{Y/X}]$ is trivial, so the element belongs to the
irrelevant part of the conormal representation. This proves statement (2) of the proposition.

Now, we instead assume that $f$ is a gerbe. Then all three squares in the diagrams in the beginning of
the proof are 2-cartesian. We have a surjective map $\Delta_\xi \to
\Delta_{f\circ\xi}$, which corresponds to an injection, which identifies $A(f\circ\xi)$ with a subgroup
of $A(\xi)$. Since conormal bundles
commute with flat base change, we get
$$
[\Nbl{\BB\Delta_\xi/Y}] = a^\ast[\Nbl{\BB\Delta_{f\circ\xi}/X}].
$$
The pull-back functor $a^\ast$ corresponds to restriction of representations. Dually,
this means that grading is preserved. The same argument applies to the second diagram, so
the grading is preserved in the appropriate way also for the divisorial part.
This proves statement (3).
\end{proof}

\begin{prop}
\label{prop-conormal-semi}
Let $\iota\colon U \to W$ be a conormal invariant, and let $(X, \bs{E})$
be a standard pair with diagonalisable stabilisers. Then the realisation
$\iota_{(X, \bs{E})/S}$ is an upper semi-continuous function. In particular,
the locus where $\iota_{(X, \bs{E})/S}$ obtains a maximum is a closed subset of $|X|$.
\end{prop}
\begin{proof}
Since the property of being semi-continuous is preserved
under post composition by order-preserving functions, it is enough to verify
semi-continuity for the universal conormal invariant. Furthermore,
it is enough to verify that the locus where $u_{(X, \bs{E})/S}$ obtains a
maximum $m \in U(\bs{E})$ is closed.

Let $\xi \in X$ be a point such that $u_{(X, \bs{E})/S} = m$, and let
$V = V_1 \oplus\cdots\oplus V_s$ be the conormal representation
in that point, with $V_i$ having degree $a_i \in A(\xi)$. Let
$E = E^1 + \cdots + E^r$ be the decomposition of $E$ in its components, 
and let $J \subset \{1, \ldots, r\}$ be the set of indices such that
$\xi \in E^i$ precisely when $i \in J$. Choose local homogeneous
coordinates
$$
(\sheaf{O}_S[x_1, \ldots, x_s], f\colon X'\to X, g\colon X'\to
[\AA^n/\Delta_\xi], \xi')
$$
at $\xi$, compatible with the conormal representation. Denote
$[\AA^n/\Delta_\xi]$ by $X_0$ and the divisor corresponding to $\bs{E}$
by $\bs{E}_0$. By the explicit description of the conormal representation
given in Proposition~\ref{prop-basic-conormal}, the maximum for
$u_{(X_0, \bs{E}_0)/S}$ is obtained in the closed substack
$Z_0 = V(x_i\mid a_i \neq 0 \text{ or } i \in J)$.
Since $g$ is continuous and $u_{(X', \bs{E}')/S} = u_{(X_0/\bs{E}_0)/S} \circ|g|$
by Proposition~\ref{prop-conormal-functorial}, also the locus $Z'$ where
$u_{(X', \bs{E}')/S}$ is $m$ is closed. Since the question regarding upper
semi-continuity is Zariski-local on $X$, and $f$ is open, we may assume
that $f$ is surjective. Since also $u_{(X, \bs{E})/S}\circ|f| = u_{(X', \bs{E}')/S}$,
the locus $Z$ where $u_{(X, \bs{E})/S}$ obtains $m$ pulls back to $Z'$. It follows
that also $Z$ is closed, since $|f|$ is submersive.
\end{proof}

In the destackification algorithm we need to blow up $X$ in a locus which is maximal
with respect to some conormal invariant. Since we only want blow ups with smooth centres,
we need a criterion to ensure that the maximal locus has a structure of a smooth
substack. If the base $S$ is reduced, it is obvious that there can be at most one
such structure, but in the general case this is not so clear. Fortunately,
there exists a simple condition, which is easy to verify in practice, which ensures
both of these properties.

First we note that every pair $\alpha, \beta \in U(C)$ of universal conormal invariants
with a common upper bound $\gamma$ has greatest lower bound $\alpha \wedge \beta$.

\begin{definition}
\label{def-conormal-smooth}
Let $\iota$ be a conormal invariant. We say that $\iota$ is \term{smooth} if the
following condition is satisfied. For each totally ordered set $C$ and each triple
$\alpha, \beta, \gamma$ in $U(C)$ such that $\gamma$ dominates both $\alpha$
and $\beta$, the condition $\iota(\alpha) = \iota(\beta) = \iota(\gamma)$
implies $\iota(\alpha \wedge \beta) = \iota(\gamma)$.
\end{definition}

\begin{example}
All the invariants in Example~\ref{example-basic-conormal-invariants} are smooth
and satisfy Conditions~P1 and~P2.
\end{example}

\begin{prop}
\label{prop-smooth-conormal-invariant}
Let $(X, \bs{E})/S$ be a standard pair with diagonalisable stabilisers,
and $\iota\colon U \to W$ a smooth conormal invariant. 
Let $m$ be a maximal value for $\iota_{(X, \bs{E})/S}$. Then the locus
where $\iota_{(X, \bs{E})/S}$ obtains $m$ has a unique structure of
smooth substack of $X$ having normal crossings with $\bs{E}$.
\end{prop}
\begin{proof}
The question is local on the base, so we may assume that $S = \Spec R$
is affine. By a standard limit argument, we may also assume that $R$
is noetherian.

We start by investigating the situation locally. Let $A$ be a finite
abelian group, and let $R[x_1, \ldots, x_s]$ be a graded ring with
$x_i$ homogeneous of degree $a_i \in A$. Assume that
$X = [\Spec R[x_1, \ldots, x_s]/A^\vee]$, and that the components
$E^1, \ldots, E^r$ of the ordered divisor $\bs{E}$ passing through the origin
correspond to $V(x_1), \ldots, V(x_r)$.

Now let $\alpha$ be the value of the universal conormal invariant at the origin,
and let $\beta$ be the greatest lower bound of the set 
$\{\alpha' \mid \alpha \geq \alpha', \iota(\alpha) = \iota(\alpha')\}$.
Such an element exists since the set is finite, and it is contained in
the set by the smoothness hypothesis for $\iota$.
Let $K$ be the kernel of the group homomorphism inducing the relation
$\alpha \geq \beta$, and let $O \subset \{1, \ldots, r\}$ be the set
of indices corresponding to the divisorial part of $\beta$.
Also define the subset $P \subset \{r+1, \ldots, s\}$ for which $a_i \not\in K$.
Note that $a_i \in K$ for all $i \in \{1, \ldots, r\}\setminus O$ by
the definition of the ordering of universal conormal invariants.
From the explicit description of the conormal representation given in
Proposition~\ref{prop-basic-conormal}, it is easy to see that the maximal
locus for $\iota_{(X, \bs{E})/S}$ corresponds to
$Z = V(x_i  \mid i \in O \cup P)$. In particular $Z$ is smooth.
The locus $F = V(x_i \mid i \in O)$ is the intersection of the divisors
$E^i$ containing $Z$.

If $S$ is reduced, the substack $Z \subset X$ is clearly the only
substack structure on the underlying space $|Z|$ of the required form.
If $S$ is non-reduced, we can, by the noetherian
hypothesis, factor the map $S_\red \hookrightarrow S$ into a finite sequence
of nilpotent thickenings defined by square zero ideals. It is enough to show that
the substack structure of $Z_\red$ lifts uniquely at each step. This
reduces the situation to the following deformation problem:
$$
\xymatrix{
Z \ar@{.>}[r] 
& F \ar@{^{(}->}[r]
& X \ar[r]
& S
\\
Z_0 \ar@{^{(}->}[r] \ar@{.>}[u]
& F_0 \ar@{^{(}->}[r] \ar[u]
& X_0 \ar[r] \ar[u]
& S_0 \ar@{^{(}->}[u]
\\
}
$$
where the map $S_0 \to S$ is a nilpotent thickening defined by a square zero
ideal $J$. We want to show that the stack $Z$, together with the dashed arrows,
is essentially the only stack fitting into the diagram, in a way such that
the leftmost square becomes cartesian and the stack becomes smooth over $S$.
Note that, since we require $Z$ to have normal crossings with $\bs{E}$,
we deform $Z$ inside $F$ and not inside $X$.
Let $I$ be the ideal $\langle x_i \mid i \in P \rangle$ in the homogeneous
coordinate ring of $F_0$, and let $\sheaf{I}$ be the corresponding ideal in
$\sheaf{O}_{F_0}$. Let $\sheaf{M}$ be the sheaf of $\sheaf{O}_{Z_0}$-modules
$\sHom_{\sheaf{O}_{Z_0}}\left(\sheaf{I}/\sheaf{I}^2, J\otimes\sheaf{O}_{Z_0}\right)$.
Then the set of objects completing the diagram is a torsor under the group
$\HH^0\left(Z_0, \sheaf{M}\right)$. The sheaf $\sheaf{M}$ corresponds to the
graded $R'$-module $$
M^\bullet = \Hom^\bullet_{R'}(I/I^2, J\otimes R'),
$$
where $R'$ is the $A$-graded ring $R[x_1, \ldots, x_s]/\langle x_i \mid i \in O
\cup P\rangle$. The global sections functor factors through the pushforward functor
$\pi_\ast$, where $\pi\colon Z_0 \to (Z_0)_\coarse$ is the map to the coarse
space, and $\pi_\ast\sheaf{M}$ is simply the degree zero part of
$M^\bullet$, viewed as an $(R')^0$-module. But the homogeneous elements of
$R'$ have degrees in $K$, whereas $I$ is generated by homogeneous elements with
degrees not in $K$. It follows that the degree zero part of $M^\bullet$ is the
zero-module, which shows that the lift of $Z_0$ is unique.

Now let $X' \to X$ be an étale stabiliser preserving map. Denote the pull
backs of $Z_0$ and $Z$ by $Z'_0$ and $Z'$ respectively. The natural map
$Z_0' \to Z_0$ is also étale and stabiliser preserving.
By flatness, the sheaf $\sheaf{M}$ pulls back to the sheaf $\sheaf{M}' =
\sHom_{\sheaf{O}_{Z'_0}}\left(\sheaf{I}'/(\sheaf{I}')^2,
J\otimes\sheaf{O}_{Z'_0}\right)$, where $\sheaf{I}'$ is the ideal sheaf
defining $Z_0'$ in $X_0$. The square
$$
\xymatrix{
X' \ar[r] \ar[d] & X \ar[d]\\
(X')_\coarse \ar[r] & X_\coarse
}
$$
formed by the maps to the coarse spaces is cartesian and the horizontal maps are
étale. Hence also $\HH^0\left(Z'_0, \sheaf{M}'\right) = 0$, and we get unicity
for the closed substack $Z' \subset X'$. A general stack $(X'', E'')$ satisfying
the standard hypothesis can be covered by stacks as $X'$ above. The unicity
of $Z' \subset X'$ asserts that the stack structure descends to a closed
substack $Z''$ of $X''$ as desired.
\end{proof}

\section{Outline of the algorithm}
\label{sec-alg-outline}
In this section, we outline the destackification algorithms. We also
introduce the various conormal invariants used by the algorithms and
describe how they are used. All invariants we define, except
the \term{divisorial type}, take values among the natural numbers.
To emphasise their geometrical meaning, we describe them as functions
defined on the underlying topological space. That is, we describe the
{\em realisation} of the invariant, rather than the invariant itself. 

To avoid tedious repetitions, we fix some notation, which we use
throughout the section. As usual, we let $(X, \bs{E})/S$ be a standard
pair with diagonalisable stabilisers. We assume that $\bs{E}$ has $m$
components. Fix a point $\xi \in X$, and let
$\bs{E}(\xi) = \{E^1, \ldots,E^s\}$ be the components of $\bs{E}$
passing through $\xi$. Denote the Cartier dual of the stabiliser at
$\xi$ by $A(\xi)$. We choose the indexing such that the conormal
representation $V(\xi)$ at $\xi$ splits as
$$
\mathrlap{
  \overbrace{\phantom{
    V_1 \oplus \cdots \oplus V_s
    \oplus V_{s+1} \oplus \cdots \oplus V_t
  }}^{\text{relevant}}
  \phantom{\oplus}
  \overbrace{\phantom{
    V_{t+1} \oplus \cdots \oplus V_r,
  }}^{\text{irrelevant}}
}
\mathrlap{
  \underbrace{\phantom{
    V_1 \oplus \cdots \oplus V_s
  }}_{\text{divisorial}}
  \phantom{\oplus}
  \underbrace{\phantom{
    V_{s+1} \oplus \cdots \oplus V_t
    \oplus V_{t+1} \oplus \cdots \oplus V_r,
  }}_{\text{residual}}
}
V_1 \oplus \cdots \oplus V_s
\oplus
V_{s+1} \oplus \cdots \oplus V_t
\oplus
V_{t+1} \oplus \cdots \oplus V_r,
$$
with the subrepresentation $V_i$ corresponding to $E^i$ for $1 \leq i \leq s$.
We also let $a_i \in A(\xi)$ be the degree of $V_i$ for $1 \leq i \leq r$, so
$a_i = 0$ for $i > t$. We let $A_\div(\xi) = \langle a_1, \ldots, a_s\rangle$
be the subgroup of $A(\xi)$ generated by the degrees of the components in the
divisorial part.

\begin{definition}[Independency index]
\label{def-independency-index}
A one-dimensional component with degree $a_i$ of the conormal representation is
said to be \term{independent} provided that the intersection
$$
\langle a_i\rangle \cap \langle a_1, \ldots, \widehat{a}_i,
\ldots,a_r\rangle, $$
is the trivial subgroup. The \term{independency index at $\xi$} is the number
of components of the conormal representation which are not independent.
A component of $\bs{E}$ passing through $\xi$ is said to be \term{independent at
$\xi$} provided that the corresponding component of the conormal representation
at $\xi$ is independent. A component of $\bs{E}$ not passing through $\xi$ is
considered independent at $\xi$ by default.
\end{definition}

\noindent
The independency index measures how far the coarse space $X_\coarse$ is from
being smooth. In particular, the invariant vanishes at a point $\xi \in |X|$,
precisely when $X_\coarse$ is smooth at the corresponding point. This can
easily be seen by using local homogeneous coordinates, and using the combinatorial
characterisation in Definition~\ref{def-independency-combinatorial}.
Thus, one of the main objectives of the destackification process is to bring
this invariant to 0.

Although the independency index is a smooth conormal invariant, in the sense of
Definition~\ref{def-conormal-smooth}, it is not fruitful to just repeatedly blow
up the locus where the invariant assumes its maximum, as was demonstrated in
Example~\ref{example-toric-need-roots}. Instead we would like to use the combinatorial
approach described in Algorithm~\ref{alg-part-toric}. This requires that we have enough
globally defined divisors to work with. We introduce a conormal invariant that
quantifies this.

\begin{definition}[Toroidal index]
\label{def-toroidal-index}
The \term{toroidal index} at $\xi$ is the dimension of the residual relevant
part of the conormal representation at $\xi$. With the indexing used in the
beginning of the section, this is the number $t-s$.
If the toroidal index is zero at $\xi$, we say that the pair $(X, \bs{E})/S$
is \term{toroidal at $\xi$}. The pair $(X, \bs{E})/S$ is \term{toroidal} if it
is toroidal at each of its points.
\end{definition}

\begin{remark}
This definition of \term{toroidal stack} is closely related to the classic
definition of toroidal variety given in \cite{kempf1973}. By using local
homogeneous coordinates, it is easy to see that a stack is toroidal at a point
$\xi$ precisely when it has a toric chart at $\xi$, compatible with $\bs{E}$,
such that $\bs{E}(\xi)$ is in one-to-one correspondence with the toric divisors
of the chart. If $S$ is the spectrum of a field, this implies that
the pair $(X_\coarse, U_\coarse)$ is a toroidal variety in the sense of
{\em loc.\ cit.}, where $U_\coarse$ is the coarse space of the complement of
the support of $\bs{E}$. More generally, if $S$ is a scheme,
then $(X_\coarse, U_\coarse)/S$ is a flat family of toroidal varieties.
It should be noted that since we assume that our toric stacks are simplicial,
not every toroidal variety can be constructed in this way.
\end{remark}

\noindent
Since the toric destackification process is essentially a global
approach, some care must be taken when destackifying stacks which
are not toric, but only toroidal. This is illustrated by the following
example:

\begin{example}
\label{example-toroidal-destack}
Consider a 2-dimensional toroidal stack $(X, \bs{E})$, where $\bs{E}$ has two
components $E^1$ and $E^2$, that intersect at two points $P$ and $Q$.

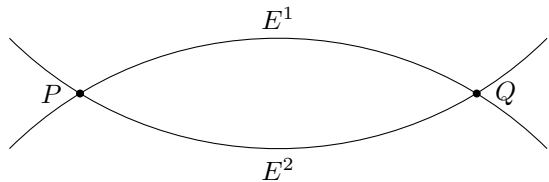
\begin{center}
\begin{tikzpicture}
\draw
  [name path = a1] ([shift=(45:5cm)]0,0) arc (45:135:5cm);
\draw
  [name path = a2] ([shift=(-45:5cm)]0,8.5355) arc (-45:-135:5cm);
\path [name intersections={of= a1 and a2}];
\fill (intersection-1) circle (1.5pt) node[right] {$\ Q$};
\fill (intersection-2) circle (1.5pt) node[left] {$P\ $};
\node [above] at (0, 5) {$E^1$};
\node [below] at (0, 3.5355) {$E^2$};
\end{tikzpicture}
\end{center}
\noindent
Assume that the independency index is 2 at $P$ and $Q$.
Clearly, we must blow up both $P$ and $Q$ during the destackification
process, but not necessarily at the same time. Locally, at each of the
points $P$ and $Q$, the stack $X$ is isomorphic to toric stacks, but
these stacks need not be isomorphic to each other. Thus it might
be necessary to apply different combinatorial recipes to destackify the
points. Even if they are isomorphic, the components $E^1$ and $E^2$ may play
different roles, so the order of the components are important.
\end{example}
\noindent
The example shows that we need an invariant which captures the
combinatorial recipe for destackification. In principle, we use the stacky cone describing
the toric stack to which $X$ is locally isomorphic at the point in question.
We shall, however, use a more algebraic description. To make the invariant
useful also in the non-toroidal case, we discard the information from the
residual part of the conormal representation. We also discard information
about the generic stabiliser and make sure that independent divisors passing
through the point have no effect.

We start by describing the ordered set where the invariant takes its values.
Consider the class of pairs $(B, \bs{b} \in B^{\bs{E}})$, where $B$ is a
finite abelian group generated by the components of $\bs{b}$. The class has a partial
preorder $\succeq$ defined by letting $(B, \bs{b}) \succeq (B', \bs{b}')$ if
there exists a surjective homomorphism $\varphi\colon B \to B'$ such that the
natural map $B^{\bs{E}} \to (B')^{\bs{E}}$ takes $\bs{b}$ to $\bs{b}'$. Let
$T(\bs{E})$ be the partially ordered set corresponding to this partial
preordering.

We will need the set $T(\bs{E})$ to be well-ordered by a relation $\geq$ in
a way compatible with the natural order $\succeq$. There are, of course, many
different ways to construct such a well-ordering. By the assumption that the
components $b_i$ generate $B$, we have a presentation
$$
\ZZ^m \to \ZZ^m \to B \to 0
$$
where the second map is the natural map and the first map is represented by 
a matrix $C = (c_{ij})$. We can choose the presentation such that $C$
becomes upper triangular with non-negative entries. We order the set $U_m(\NN)$ of
such matrices lexicographically. Here the entries are ordered first by rows, with high
row numbers being more significant, and then by columns, with low column numbers
being more significant. It is easy to verify that the map
$T(\bs{E}) \to U_m(\NN)$ taking a pair $(B, \bs{b})$ to the minimal $C$
giving a presentation of $B$ is injective and order preserving. We now
transport the well-ordering on $U_m(\NN)$ to $T(\bs{E})$.

\begin{definition}[Divisorial type]
\label{def-divisorial-type}
Define the vector $\bs{b} \in A(\xi)^{\bs{E}}$ as follows. For each $E^i \in \bs{E}(\xi)$
which is not independent at $\xi$, we let the corresponding component of $\bs{b}$ be $a_i$.
The other components of $\bs{b}$ are set to zero. Let $B$ be the subgroup of $A(\xi)$
generated by the components of $\bs{b}$. The \term{divisorial type} at $\xi$ is the element
in $T(\bs{E})$ corresponding to $(B, \bs{b})$.
\end{definition}

The techniques described so far are enough to solve the destackification problem
in the toroidal case. The procedure is described by
Algorithm~\ref{alg-destack} if we omit Step \ref{alg-destack}5.
If we do not have a toroidal structure when we start, we need to create one.
One problem is that the toroidal index is not a smooth
conormal invariant, as indicated in the following example.
\begin{example}
Consider the basic toric 3-orbifold over the field $k$ with
homogeneous coordinate ring $(k[x_1, x_2, x_3], \ZZ/2\ZZ\times\ZZ/2\ZZ, \bs{a})$,
where $a_1 = (1, 0)$, $a_2 = (0, 1)$ and~$a_3=(1, 1)$. Furthermore, assume that
$\bs{E} = \{\VV(x_1), \VV(x_2)\}$. Then the toroidal index is 1 at the locus
$\left(\VV(x_1) \cup \VV(x_2)\right) \cap \VV(x_3)$ and 0 outside this
locus. In particular, we see that the toroidal index is not a smooth conormal
invariant.
\end{example}
Instead, we introduce a coarser invariant, the \term{divisorial index}, which
may be thought of as a smoothened version of the toroidal index.
\begin{definition}[Divisorial index]
\label{def-divisorial-index}
The \term{divisorial index} at $\xi$ is the number of elements $a_i$, for $1
\leq i \leq t$ such that $a_i \not\in A_\div$. If the divisorial index is
zero at $\xi$, we say that the pair $(X, \bs{E})/S$ is \term{divisorial at
$\xi$}.
Furthermore, we say that the pair $(X, \bs{E})/S$ is \term{divisorial} if it is
divisorial at each of its points.
\end{definition}
\begin{remark}
Geometrically, the property for a stack of being divisorial, can be understood
as follows: Each of the components $E^i$ of $\bs{E}$ gives rise to a
$\GGm$-torsor $F^i$, and the fibre product $F = F^1 \times_X \cdots \times_X F^m$ is a
$\GGm^r$-torsor. If $X$ is an orbifold, the pair $(X, \bs{E})$ is divisorial
precisely when $F$ is an algebraic space. In general, the pair $(X, \bs{E})$
is divisorial precisely when $F$ is a gerbe.

Classically, a scheme is called \term{divisorial} \cite[Def.\ 2.2.5]{SGA6expII} if it
has an ample family of line bundles. This is equivalent to the
scheme having a $(\GG_m)^n$-torsor, for some $n$, whose total space is quasi-affine
(see \cite[Thm.~1]{hausen2002} for varieties and \cite[Cor.~5.5]{gross2013} for the
generalisation to stacks). Hence our notion of divisorial stack is related, but not
equivalent, to the classical definition.
\end{remark}

\noindent
The process of modifying $X$ such that it becomes divisorial is
straightforward, and described in Algorithm~\ref{alg-divisorialification}.
But to modify a divisorial stack such that it becomes toroidal is trickier.
It turns out that, in general, this is not possible by just using
ordinary blow-ups; root stacks are needed.
The easiest way seem to interleave the process of reducing the
toroidal index with the process of reducing the independency index.
Simply put, we just ignore the fact that $(X, \bs{E})$ is not toroidal,
and use exactly the same algorithm as in the toroidal case. The
distinguished divisors we create will, in general, not be
independent in this case, but they will have a weaker property.  

\begin{definition}
\label{def-divisorially-independent}
We assume that $(X, \bs{E})/S$ is divisorial. Let $E^i$ be a component of
$\bs{E}(\xi)$. We say that $E^i$ is \term{divisorially independent} at $\xi$
provided that the intersection
$\langle a_i\rangle \cap \langle a_1, \ldots, \widehat{a}_i, \ldots, a_s\rangle$
is the trivial subgroup. A component of $\bs{E}$ not passing through $\xi$ is
considered divisorially independent at $\xi$ by default.
\end{definition}

\noindent
Note that the property for a component of $\bs{E}$ of being independent at a
point $\xi \in X$ does not depend on the divisorial structure, but the property of
being divisorially independent does.

This reduces the problem to modifying $(X, \bs{E})$ such that the divisorially
independent divisors become independent. This is achieved by
Algorithm~\ref{alg-divisorial-dist}. Or rather, the algorithm ensures that
either the divisor becomes independent or the toroidal index drops. In either
case we get an improvement, which allows us to solve the problem by repeating
the procedure.

The main invariant used by Algorithm~\ref{alg-divisorial-dist} is slightly
more subtle than the others.

\begin{definition}[Divisorial index along a divisor]
\label{def-divisorial-index-divisor}
We assume that $(X, \bs{E})/S$ is divisorial. Let $E^i \in \bs{E}(\xi)$,
and define $A^{E^i}(\xi)$ as the quotient of
$A_\div(\xi)$ by the group $\langle a_1, \ldots, \widehat{a}_i, \ldots,
a_s\rangle$.
Let $\varphi\colon A_\div(\xi) \to A^{E^i}(\xi)$ denote the natural surjection.
The group $A^{E^i}(\xi)$ is
cyclic and generated by $\varphi(a_i)$. For each $1 \leq k \leq r$, we let
$c_k$ denote the minimal natural number such that $\varphi(a_k) =
c_k\varphi(a_i)$. Note that the coefficients $c_k$ depend on an arbitrary choice
of indexing of the residual part of the conormal representation.
They are therefore not well-defined functions on $|X|$. But the sum
$$
c = \sum_{k = s+1}^r c_k
$$
is well-defined. We call this sum the \term{divisorial index at $\xi$ along
$E^i$}.
Note that the sum is taken over the elements in the residual part of the conormal
representation. Sometimes, it is more convenient to take the sum over all indices
$1 \leq k \leq r$. It is easy to see that we have $c = \sum_{k = 1}^r c_k - 1$.
\end{definition}

Using the same notation as in the definition, the divisorial index
along a divisor $E^i$ measures how far the pair
$\left(E^i, \{E^j_{|E^i} \mid j \neq i\}\right)$ is from
being divisorial. This motivates the term \term{divisorial index along $E^i$}.
Note, however, the slight asymmetry in the terminology; the divisorial index
along $E^i$ is not quite the same as the divisorial index of
$\left(E^i, \{E^j_{|E^i} \mid j \neq i\}\right)$. Although Algorithm~\ref{alg-divisorial-dist}
would work also if we used that invariant, we would still have to introduce a
finer invariant in order to prove correctness.

We conclude the section by summarising the properties of the conormal invariants
introduced here. We leave the proof to the reader.
\begin{prop}
All of the following conormal invariants satisfy the properties P1
and~P2 defined in Section~\ref{sec-conormal}:
\begin{enumerate}[label=(\alph*), ref=\emph{\alph*}]
\item \label{ci-ii} Independency index
\item \label{ci-ti} Toroidal index
\item \label{ci-dt} Divisorial type
\item \label{ci-di} Divisorial index
\item \label{ci-did} Divisorial index along a divisor.
\end{enumerate}
Of these, the invariants \ref{ci-ii}, \ref{ci-di} and~\ref{ci-did} are smooth.
The invariants \ref{ci-ti} and~\ref{ci-dt} are not smooth, but they are locally
constant on the locus where the independency index obtains its maximum.
\end{prop}

\section{The destackification algorithm}
\label{sec-general-destack}
In this section, we give the full details of the destackification algorithms
outlined in the previous section. We fix some notation which will be used
throughout the section.

Let $(X_0, \bs{E}_0)/S$ be a standard pair with diagonalisable stabilisers over a quasi-compact
scheme $S$.  We wish to construct a smooth, stacky blow-up sequence
$$
\Pi\colon (X_n, \bs{E}_n) \to \cdots \to (X_0, \bs{E}_0).
$$
Each stacky blow-up $\pi_i\colon (X_{i+1}, \bs{E}_{i+1}) \to (X_i, \bs{E}_i)$ in the
sequence, will have a centre $Z_i$ which is determined by some smooth conormal
invariant $\iota$ as described in Proposition~\ref{prop-smooth-conormal-invariant}.
Different invariants $\iota$ will be used at different stages of the algorithm.

We wish to describe how the conormal invariant changes for each blow-up at each
point. We let $\xi_i \in X_i$ denote a point and $\xi_{i+1}$ a lifting of $\xi_i$
to $X_{i+1}$. As usual, we let $A(\xi_i)$ be the character group of the stabiliser
at $\xi_i$,
$$
V(\xi_i) = V_1(\xi_i) \oplus \cdots \oplus V_r(\xi_i)
$$
a decomposition of the conormal representation at $\xi_i$ and $a_j(\xi_i) \in A(\xi_i)$
the degree of $V_j(\xi)$ for $1 \leq j \leq r$. The decomposition will be compatible with
$\bs{E}_i$, but we will not follow the indexing convention described in the beginning of the
previous section.

Let $Z \subset X$ be the locus where $\iota_{(X_i, \bs{E}_i)/S}$ obtains its maximum.
The sheaf $\sheaf{N}_{Z/X}$ determines a subrepresentation of $V(\xi_i)$. A component
$V_j(\xi_i)$ contained in this subrepresentation is called \term{critical} for $\iota$.
We denote the set of indices for the critical components by $J(\xi_i) \subset \{1, \ldots, r\}$.

\begin{lemma}
\label{prop-conormal-blowup}
Let $(X_i, \bs{E}_i)/S$ be a standard pair and fix a smooth conormal invariant $\iota$.
Let $Z_i$ be the substack where $\iota_{(X_i, \bs{E}_i)}$ obtains its maximum, and let
$\xi_i \in Z_i$. Let $\pi_i\colon (X_{i + 1}, \bs{E}_{i+1}) \to (X_i, \bs{E}_i)$ be the
blow-up with centre $Z_i$ and let $\xi_{i+1}$ be a lifting of $\xi_i$ to $X_{i+1}$
where the restriction of the \term{universal conormal invariant} $u_{(X_{i + 1},\bs{E}_{i+1})}$
to the fibre of $\xi_i$ obtains a maximum. Use the notation from the introduction of the section.
Then there exists an index $p \in J(\xi_i)$, and a homogeneous decomposition,
compatible with $\bs{E}_{i+1}$,
$
V(\xi_{i+1}) = V_1(\xi_{i+1}) \oplus \cdots \oplus V_r(\xi_{i+1})
$
of the conormal representation at $\xi_{i+1}$ satisfying the following properties:
\begin{enumerate}
\item[(i)] The canonical map $A(\xi_i) \to A(\xi_{i+1})$ is an isomorphism.
In the sequel, we will identify these groups.
\item[(ii)] The degree $a_j(\xi_{i+1})$ of $V_j(\xi_{i+1})$ equals $a_j(\xi_i) - a_p(\xi_i)$
if $j \in J(\xi_i) \setminus \{p\}$, and $a_j(\xi_i)$ otherwise.
\item[(iii)] The component $V_p(\xi_{i+1})$ is marked by the exceptional divisor
of the blow up.
\item[(iv)]
Assume that $j \neq p$. If $V_j(\xi_i)$ is divisorial, then the component
$V_j(\xi_{i+1})$ is divisorial and marked by the strict transform of the
component marking $V_j(\xi_i)$. Otherwise, the component $V_j(\xi_{i+1})$ is
residual.
\end{enumerate}
\end{lemma}
\begin{proof}
By the functoriality properties stated in Proposition~\ref{prop-conormal-functorial},
we can pass to local homogeneous coordinates $\sheaf{O}_S[x_1, \ldots, x_r]$ of $X_i$
compatible with $E_i$ and $Z_i$. In these coordinates $Z_i$ is the closed
substack corresponding to $V(x_j \mid j \in J(\xi_i))$. We have a covering, consisting of
$|J(\xi_i)|$ patches, of the blow up $X_{i+1}$. The patch corresponding to $p \in J(\xi_i)$
has homogeneous coordinate ring
$
\sheaf{O}_S[y_1, \ldots, y_r],
$
where $y_j = x_j/x_p$ if $j \in J(\xi_i) \setminus \{p\}$, and $x_j$ otherwise. From this,
statements (i) and (ii) follow. The map $X_{i+1} \to X_i$ corresponds to the graded ring map
$\sheaf{O}_S[x_1, \ldots, x_r] \to \sheaf{O}_S[y_1, \ldots, y_r]$ given by
$x_j \mapsto y_jy_p$ if $j \in J(\xi_i)\setminus\{p\}$ and $x_j \mapsto y_j$ if
$j \not\in J(\xi_i)\setminus\{p\}$. From this, statements (iii) and (iv) follow easily.
\end{proof}

\begin{algorithm}[Divisorialification]
\label{alg-divisorialification}
The input of the algorithm is a standard pair $(X, \bs{E})$ over a quasi-compact scheme
$S$, with $X$ having diagonalisable stabilisers. The output of the algorithm is a smooth,
ordinary blow-up sequence
$$
\Pi \colon (X_n, \bs{E}_n) \to \cdots \to (X_0, \bs{E}_0) = (X, \bs{E}),
$$
such that $(X_n, \bs{E}_n)/S$ is divisorial. The construction is functorial with respect
to gerbes, smooth, stabiliser preserving morphisms and arbitrary base change.
\begin{enumerate}
\itemsep0em
\item[\bf \thealgorithm0.][Initialise.] Set $i = 0$.
\item[\bf \thealgorithm1.][Finished?] Let $Z_i$ be the locus in $X_i$ where the
divisorial index with respect to $\bs{E}_i$ is maximal. If $Z_i = X_i$, then the
algorithm terminates.
\item[\bf \thealgorithm2.][Blow up.] Let $(X_{i + 1}, E_{i+1}) \to (X_i, E_i)$ be
the blow up of $X_i$ in $Z_i$.
\item[\bf \thealgorithm3.][Iterate.] Increment $i$ by 1 and iterate from
Step~\thealgorithm1.
\end{enumerate}
\end{algorithm}
\begin{proof}[Proof of correctness of Algorithm~\ref{alg-divisorialification}]
We will show that the maximum of the divisorial index decreases strictly after
each iteration of the algorithm. This cannot continue forever, so the
algorithm eventually halts. Since the divisorial index is generically 0, it
must be identically zero when the algorithm halts, which proves that
$(X_n, \bs{E}_n)$ is indeed divisorial. The functoriality properties follows directly
from the corresponding functoriality properties of conormal invariants described
in Proposition~\ref{prop-conormal-functorial}.

Assume that we are in Step~\ref{alg-divisorialification}2, and let $\xi_i$ be a
point in $Z_i$. Let $V(\xi_i) = V_1(\xi_i) \oplus \cdots \oplus V_r(\xi_i)$ be
the conormal representation at $\xi_i$ in $X_i$, and let $a_j(\xi_i) \in A(\xi)$
denote the weight of $V_j(\xi_i)$. Denote the subgroup of $A(\xi_i)$ generated
by the weights of the marked components of $V$ by $A_\div(\xi_i)$. It follows
directly from the definition of the divisorial index
(\ref{def-divisorial-index}), that the subset $J(\xi_i) \subseteq \{1, \ldots r\}$
of indices corresponding to critical components $V_j(\xi_i)$ are precisely the
set of indices such that $a_j(\xi_i) \not\in A_\div(\xi_i)$, and that the
divisorial index is simply the cardinality $|J(\xi_i)|$.

Now consider the blow up $(X_{i + 1}, \bs{E}_{i+1}) \to (X_i, \bs{E}_i)$
described in Step~\ref{alg-divisorialification}3. We choose a point
$\xi_{i+1} \in X_{i+1}$ lying over $\xi_i$, which is maximal in the sense
described in the statement of Proposition~\ref{prop-conormal-blowup}.
Also let $p \in J(\xi_i)$, and $V(\xi_{i+1}) = V_1(\xi_{i+1}) \oplus \cdots
\oplus V_r(\xi_{i+1})$ be as in that proposition.

If $j$ is an index corresponding to a component $V_j(\xi_i)$ in the divisorial
part, then $j \not\in J(\xi_i)$. Hence $a_j(\xi_{i+1}) = a_j(\xi_i)$. Since
$V_j(\xi_{i+1})$ also is in the divisorial part, we have $A_\div(\xi_i) \subseteq
A_\div(\xi_{i+1})$. It follows that $J(\xi_{i+1}) \subseteq J(\xi_i)$, where
$J(\xi_{i+1})$ is the set of indices for the critical components for
$V(\xi_{i+1})$. But since $V_p(\xi_{i+1})$ is a marked component of
$V(\xi_{i+1})$, we have $p \not\in J(\xi_{i+1})$. Since $p \in J(\xi_i)$,
this shows that the inclusion is strict, so the divisorial index has decreased
strictly.
\end{proof}

\begin{algorithm}[Divisorialification along distinguished divisors.]
\label{alg-divisorial-dist}
The input of the algorithm is a divisorial stack with distinguished
structure $(X, \bs{E}, \bs{D})$ over a quasi-compact base scheme $S$. The output of the
algorithm is a smooth, ordinary blow up sequence
$$
\Pi \colon (X_n, \bs{E}_n, \bs{D}_n) \to \cdots \to (X_0, \bs{E}_0, \bs{D}_0) = (X, \bs{E}, \bs{D}),
$$
over $S$, such that the divisorial index of $(X_n, \bs{E}_n)$ vanishes along all
components of $\bs{D}_n$. Furthermore, each of the centres $Z_i$ in the blow up
sequence is contained in exactly one of the components of $\bs{D}_i$ and transversal
to all other components of $\bs{E}_i$. The construction is functorial with respect to
gerbes, smooth, stabiliser preserving morphisms and arbitrary base change.
\begin{enumerate}
\itemsep0em
\item[\bf\thealgorithm 0.][Initialise.]
Let $i = 0$.
\item[\bf\thealgorithm 1.][Finished?]
Let $D'_i$ denote the oldest component of $\bs{D}_i$ for which the divisorial
index along $D'_i$ does not vanish identically. If no such component exists,
the algorithm terminates. Otherwise, we let $Z_i$ be the smooth substack
of $X_i$ where the divisorial index along $D'_i$ assumes its maximal value.
\item[\bf\thealgorithm 2.][Blow up maximal locus.]
Let $(X_{i+1}, \bs{E}_{i+1}, \bs{D}_{i+1}) \to (X_i, \bs{E}_i, \bs{D}_i)$ be the blow up of
$X_i$ in $Z_i$.
\item[\bf\thealgorithm 3.][Iterate.]
Increment $i$ by one and iterate from Step~\thealgorithm1.
\end{enumerate}
\end{algorithm}
\begin{proof}[Proof of correctness of Algorithm~\ref{alg-divisorial-dist}]
Assume that we are in iteration $i$. Let $\xi_i \in Z_i$ be a point in the
centre of the blow up, and let $\xi_{i+1}$ be an arbitrary lifting of $\xi_i$
to the exceptional locus. We will prove the following three statements:
\begin{enumerate}
\item
The divisorial index at $\xi_{i+1}$ along each of the components of the
strict transform of $\bs{D}_i$ is not larger than the divisorial index
at $\xi_i$ along the corresponding component of $\bs{D}_i$.
\item
In the case where the component in the previous statement is $D'_i$, the
index is strictly smaller.
\item
The divisorial index at $\xi_{i+1}$ along the exceptional divisor of the blow up
is strictly smaller than the divisorial index at $\xi_i$ along $D'_i$.
\end{enumerate}
Together these statements prove that the algorithm terminates with the right
exit condition. Indeed, let $N$ be the maximum of the divisorial indices
along all distinguished divisors when the algorithm starts, and let
$w_j$ be the number of components of $D_i$ such that the maximum of the
divisorial index along the component is $j$. Then the $N$-tuple
$(w_N, w_{N-1}, \ldots, w_1)$ decreases strictly in lexicographical
ordering with each iteration of the algorithm.

Let $V(\xi_i) = V_1(\xi_i) \oplus \cdots \oplus V_r(\xi_i)$, $a_j(\xi_i) \in
A(\xi_i)$, $J(\xi_i)$, $\xi_{i+1}$, $p \in J(\xi_i)$, $V(\xi_{i+1}) =
V_1(\xi_{i+1}) \oplus \cdots \oplus V_r(\xi_{i+1})$ and $a_j(\xi_{i+1}) \in
A(\xi_{i+1})$ be as in Proposition~\ref{prop-conormal-blowup}.
It is enough to prove the three statements for the point $\xi_{i + 1}$ for
various $p$. We choose the indexing of the components of $V(\xi_i)$ such
that $V_1(\xi_i)$ corresponds to $D'_i$ as defined in
Step~\ref{alg-divisorial-dist}1.
For each $k$, such that $1 \leq k \leq r$, we define the subgroup
$$
A^k(\xi_i) = \langle a_j(\xi_i) \mid j \neq k, V_j(\xi_i) \text{ is
divisorial}\rangle,
$$
of $A_\div(\xi_i)$, and for each $j$ such that $1 \leq j \leq r$, we
define $c_j^k(\xi_i)$ as the smallest natural numbers such that
$a_j(\xi_i) \equiv c_j^k(\xi_i)a_k(\xi_i) \mod A^k(\xi_i)$. Recall from
Definition~\ref{def-divisorial-index-divisor} that if $V_k(\xi_i)$ corresponds
to a component of $E$, the divisorial index along that component is the
sum $c_1^k(\xi_i) + \cdots + c_r^k(\xi_i) - 1$.

Now assume that $V_k(\xi_i)$ corresponds to a component of the distinguished
divisor. If $k \neq p$, then $V_k(\xi_{i+1})$ corresponds to the strict
transform of that divisor. Note that $A^k(\xi_{i+1}) = A^k(\xi_i) + \langle
a_p(\xi_i)\rangle$. It follows that for any $j$ such that $1 \leq j \leq r$, we
have
$$
a_j(\xi_{i+1}) \equiv c_j^k(\xi_i)a_k(\xi_{i+1}) \mod A^k(\xi_{i+1}),
$$
so $c^k_j(\xi_{i+1}) \leq c^k_j(\xi_i)$. By taking the sum over all $j$,
we see that the divisorial index along the component corresponding to
$V_k(\xi_i)$ has not increased, which proves the first statement. In the
particular case when $k = 1$, one verifies that $c_p^1(\xi_{i+1}) = 0$, whereas
$c_p^1(\xi_i) \neq 0$, which proves the second statement.

Finally, we investigate the divisorial index at $\xi_{i+1}$ along the
exceptional divisor, which corresponds to the component $V_p(\xi_{i+1})$.
We have $A^p(\xi_{i+1}) = A^1(\xi_i) + \langle a_1(\xi_i) - a_p(\xi_i) \rangle$,
so $a_p(\xi_{i+1})\equiv a_p(\xi_i)\equiv a_1(\xi_i) \mod A^p(\xi_{i+1})$.
For $j \in J(\xi_i)$, we get
$$
a_j(\xi_{i+1}) \equiv a_j(\xi_i) - a_p(\xi_i) \equiv (c_j^1(\xi_i)-1)a_p(\xi_{i+1}) \mod A^p(\xi_{i+1}),
$$
which proves that $c_j^p(\xi_{i+1}) < c_j^1(\xi_i)$, and the third statement
follows.
\end{proof}

\begin{lemma}
\label{lemma-divisorial-independence}
Let $(X, \bs{E}, \bs{D})$ be a divisorial stack with distinguished structure over a
quasi-compact scheme $S$, and let
$$
\Pi \colon (X_n, \bs{E}_n, \bs{D}_n) \to \cdots \to (X_0, \bs{E}_0, \bs{D}_0) = (X, \bs{E}, \bs{D}),
$$
be the output of Algorithm~\ref{alg-divisorial-dist} applied to $(X, \bs{E}, \bs{D})$.
Let $\xi \in X$ be a point at which all distinguished divisors are divisorially
independent with respect to $\bs{E}$, and let $\xi_i$, for $i$ such that $0 \leq i \leq n$,
be a lift of $\xi$ to $X_i$ such that $\xi_i$ has the same toroidal index as $\xi$.
Then all components of $\bs{D}_i$ are divisorially independent at $\xi_i$ with respect to $\bs{E}_i$.
In particular, all components of $\bs{D}_n$ are independent at $\xi_n$.
\end{lemma}
\begin{proof}
Note that a divisor is independent at a point if and only if it is divisorially
independent and the divisorial index along the divisor in question is zero. This
is an easy consequence of the definitions. Therefore the last statement of the lemma
follows from the second last statement.

We fix $i$ and assume that all components of $\bs{D}_i$ are divisorially independent
at $\xi_i$. We want to prove that all components of $\bs{D}_{i+1}$ are divisorially
independent at $\xi_{i+1}$.  

We may, without loss of generality, assume that $\xi_i \in Z_i$. Assume
that $\xi_{i + 1}$ is maximal with respect to the universal conormal
representation in the sense described in the statement of
Proposition~\ref{prop-conormal-blowup}. We use the same notations as in the
proof of correctness for Algorithm~\ref{alg-divisorial-dist}. Since the toroidal index
at $\xi_{i+1}$ is assumed to be the same as the toroidal index at $\xi_i$,
this forces $p = 1$. It follows that the divisorial components of the decomposition
$$
V(\xi_{i+1}) = V_1(\xi_{i+1}) \oplus \cdots \oplus V_r(\xi_{i+1})
$$
of the conormal representation at $\xi_{i+1}$ have the same weights as the
corresponding components of the corresponding decomposition of $V(\xi_i)$.
Furthermore, a component of $V(\xi_{i+1})$ is distinguished if and only if
the corresponding component of $V(\xi_i)$ is. In particular, all distinguished
divisors are divisorially independent at $\xi_{i+1}$. The same holds if
$\xi_{i+1}$ is not maximal with respect to the universal conormal
representation, provided that the toroidal index at $\xi_{i+1}$ is the
same as that of $\xi_i$, which proves the lemma.
\end{proof}

\begin{algorithm}[Destackification.]
\label{alg-destack}
The input of the algorithm is a divisorial stack $(X, \bs{E})$ over a quasi-compact
base scheme $S$. The output of the algorithm is a smooth stacky blow-up sequence
$$
\Pi \colon (X_n, \bs{E}_n) \to \cdots \to (X_0, \bs{E}_0) = (X, \bs{E}),
$$
such that the independency index is everywhere 0 at $X_n$. The construction is functorial
with respect to gerbes, smooth stabiliser preserving morphisms and arbitrary base change.
\begin{enumerate}
\itemsep0em
\item[\bf\thealgorithm 0.][Initialise.]
Let $i = 0$.
\item[\bf\thealgorithm 1.][Find the worst locus.]
Consider the aggregate conormal invariant, composed by the
independency index, the toroidal index and the divisorial type.
Let $Z_i$ be the locus in $X_i$ where the realisation of this invariant
obtains its maximum. If $Z_i = X_i$ the algorithm terminates.
\item[\bf\thealgorithm 2.][Blow up $Z_i$.]
Let $\mathbf{\Sigma}$ be the stacky fan generated by the single stacky cone
corresponding to the divisorial type at $Z_i$. The rays in $\mathbf{\Sigma}(1)$
have a natural correspondence with the non-independent components of $E_i$
intersecting $Z_i$. In particular, this induces an ordered structure
on $\mathbf{\Sigma}$.
Let $(X_{i + 1}, \bs{E}_{i+1}) \to (X_i, \bs{E}_i)$ be the blow up of $(X_i, \bs{E}_i)$ in the
centre $Z_i$. Denote the exceptional divisor by $\bs{D}_{i + 1}$ and mark it as a distinguished
divisor. Also let $\mathbf{\Sigma}_{i + 1}$ be the star subdivision
$\mathbf{\Sigma}^\ast(\sigma)$, where $\sigma$ is the maximal cone in $\mathbf{\Sigma}$.
We label the exceptional ray $\delta$ by $\bs{D}_{i + 1}$, and give $\mathbf{\Sigma}_{i + 1}$
a distinguished structure by letting $\{\delta\}$ be the set of distinguished rays.
Increment $i$ by 1.
\item[\bf\thealgorithm 3.][Perform toric destackification.]
Perform Algorithm~\ref{alg-part-toric} on $\mathbf{\Sigma}_i$, and denote the
result by $\mathbf{\Sigma}_{i + k} \to \cdots \to \mathbf{\Sigma}_{i}$.
\item[\bf\thealgorithm 4.][Perform corresponding stacky blow ups.]
Perform the corresponding stacky blow ups on $(X_i, \bs{E}_i, \bs{D}_i)$ to form the
sequence
$$
(X_{i + k}, \bs{E}_{i + k}, \bs{D}_{i + k}) \to \cdots \to (X_i, \bs{E}_i, \bs{D}_i).
$$
At each step, a star subdivision corresponds to blow up in the intersection
of the divisors labelling the rays of the subdivided cone. A root construction
corresponds to a root stack of the same order at the corresponding ray.
Also, the ray -- divisor correspondence is extended in each step
such that the exceptional rays correspond to the exceptional divisors.
Increment $i$ by $j$.
\item[\bf\thealgorithm 5.][Eliminate divisorial index along distinguished
divisors.]
Perform Algorithm~\ref{alg-divisorial-dist} on the triple $(X_i, \bs{E}_i, \bs{D}_i)$,
and append the output of the algorithm to the blow-up sequence. Increment $i$
by the length of the output. After this step we forget the distinguished
structure.
\item[\bf\thealgorithm 6.][Iterate.]
Iterate from Step~\thealgorithm1.
\end{enumerate}
\end{algorithm}
\begin{proof}[Proof of correctness of Algorithm~\ref{alg-destack}]
When the algorithm terminates, the independence index is constant zero at
$X_i$, according to the termination criterion in Step~\thealgorithm1. Hence we
need to prove that the algorithm terminates.

We will prove that the maximum of the conormal invariant composed by the
independency index, the toroidal index and the divisorial type decreases
strictly with each iteration of the main loop of the algorithm. Note that
the independency index decreases weakly with each iteration, since all
blow ups performed during the iteration have centres which intersect the
transforms of the divisors which where independent at the start of the
iteration transversally. The toroidal index decreases weakly at each blow-up.

We examine how the invariant described above is affected during a single
iteration of the main loop. For notational convenience, we assume that $i = 0$
at the start of the iteration and that $i = n$ at the end of the iteration.
Assume that we are in Step~\thealgorithm1 of the algorithm, and 
let $\xi$ be any point in $Z_0$. Since all blow-ups during a single iteration
have centres lying above $Z_0$, it is enough to show that any point in $X_n$ lying
over $\xi$ has either strictly lower independence index, or strictly lower toroidal
index than $\xi$ at the end of the iteration. This can be verified on local
homogeneous coordinates, since both invariants are preserved by étale, stabiliser
preserving maps. Let $A_0 = A(\xi)$ be the character group of the geometric
stabiliser at $\xi$. We may, with out loss of generality, assume that $X_0$
is the stack corresponding to the $A_0$-graded coordinate ring
$$
R_0 = \sheaf{O}_S[x_1, \ldots, x_r]
$$
We choose the indexing such that the coordinates $x_1, \ldots, x_s$ correspond
to components of $E_0$ which are not independent at $\xi$, and $x_{s+1},
\ldots, x_t$ correspond to the relevant residual part of the conormal
representation at $\xi$. In particular, the independence index at
$\xi$ is $t$, since the residual components are not independent, by the
assumption that $X_0$ is divisorial. The toroidal index is $t - s$.
Also consider the subring
$$
R'_0 = \sheaf{O}_S[x_1, \ldots, x_s].
$$
This ring comes with a natural grading by the subgroup $A'_0$ of
$A_0$ generated by the degrees of the variables $x_1, \ldots, x_s$.
The stack corresponding to this graded ring is the toric stack corresponding
to the stacky fan $\mathbf{\Sigma}$.

After the blow up in Step~\thealgorithm2, the stack $X_1$ is covered by $t$
patches. Let $U_1$ be one such patch, and denote its homogeneous coordinate ring
by $(R_1, A_1)$. Explicitly, we have $A_1 = A_0$ and
$$
R_1 = \sheaf{O}_S\left[\frac{x_1}{x_j}, \ldots, \frac{x_s}{x_j}, x_j, \frac{x_{s
+ 1}}{x_j}, \ldots, \frac{x_t}{x_j}, x_{t+1}, \ldots, x_r\right]
$$
for some $j$ such that $1 \leq j \leq t$. For $j > s$, the toroidal index is
strictly lower than the toroidal index at $\xi$, so we may assume that
$j \leq s$. There is a corresponding patch $U'_1$ of the toric stack
corresponding to $\mathbf{\Sigma}_{1}$. Let $(R'_1, A'_1)$ be its
homogeneous coordinate ring. Explicitly, we have $A_1' = A_0'$ and
$$
R'_1 = \sheaf{O}_S\left[\frac{x_1}{x_j}, \ldots, \frac{x_s}{x_j}, x_j\right].
$$
The graded ring homomorphism $(R'_1, A'_1) \to (R_1, A_1)$ corresponds to a
smooth morphism $U_1 \to U'_1$ of stacks.

Now execute the partial toric resolution in Step~\thealgorithm3, and assume
that it finishes in $k$ steps. Let $U'_{k+1}$ be a patch corresponding to a
maximal cone in $\sigma \in \mathbf{\Sigma}_{k+1}$, and denote its
$A'_{k+1}$-graded homogeneous coordinate ring by
$$
R'_{k+1} = \sheaf{O}_S[y_1, \ldots, y_s].
$$
Let $U_{k+1} \to U_1$ be the pull-back of $U'_{k+1} \to U'_1$ along
$U_1 \to U'_1$. Since the latter morphism is smooth, this gives a patch
of the blow-up sequence $X_{k+1} \to X_1$. Furthermore, the locus of
$X_{k+1}$ where the toroidal index is the same as the toroidal index at $\xi$,
can be covered by such patches. Let $(R_{k+1}, A_{k+1})$ be the homogeneous
coordinate ring for $U_{k+1}$. Explicitly, the group $A_{k+1}$ is the
push-out $(A'_{k+1} \oplus A_1)/A'_1$, and the ring $R_{k+1}$ is the
tensor product
$$
R_{k+1} = \sheaf{O}_S\left[y_1, \ldots, y_s, \frac{x_{s +
1}}{x_j}, \ldots, \frac{x_t}{x_j}, x_{t+1}, \ldots, x_r\right].
$$
The group generated by the degrees of $y_1, \ldots, y_s$ is
$A'_{k+1}$ considered as a subgroup of $A_{k+1}$ via the natural
inclusion. Let $y_j$ be a coordinate function in $R'_{k+1}$ that
corresponds to a distinguished ray in $\mathbf{\Sigma}_{k+1}$.
By the exit condition of Algorithm~\ref{alg-part-toric}, the
divisor $V(y_j)$ is independent. By the explicit description of the
coordinate ring for $U_{k+1}$, we see that the pull-back of $V(y_j)$
is divisorially independent in the patch $U_{k+1}$. We conclude that
after Step~\thealgorithm4 of the algorithm, all distinguished divisors
are divisorially independent at points where the toroidal index has not
dropped.

Now execute the sub algorithm in Step~\thealgorithm5. From
Lemma~\ref{lemma-divisorial-independence}, we see that after this step
all distinguished divisors are independent at points where the toroidal index
has not dropped. Due to the first blow-up, there is at least one
distinguished divisor going through every point lying over $xi$. Thus, at
points where the toroidal index has not dropped, the independency index is
at most $t-1$.
\end{proof}

Algorithm~\ref{alg-destack} almost, but not quite, produces a destackification of
the pair $(X, \bs{E})$. It produces a stack with a smooth coarse space, but the
coarse map need not have a factorisation as described in
Definition~\ref{def-destackification}.

To describe the problem, we first introduce some extra terminology. We use the
term \term{generic order} at a point $\xi \in X$ to describe the generic order
of the stabiliser near $\xi$. The \term{relative generic order along a component}
$E^i$ of $\bs{E}$ at $\xi \in E^i$ is defined as the generic order at $\xi$ viewed
as a point in $E^i$ divided by the generic order at $\xi$ in $X$.
With the notation used in the proof of correctness for Algorithm~\ref{alg-divisorial-dist},
the relative generic order along $E^i$ is just $|A_\div(\xi)/A^k(\xi)|$, where $k$ is
the index corresponding to $E^i$ in the decomposition of the conormal representation.

It is easy to see that the relative generic order along $E^i$ is locally
constant on $E^i$. But since we do not require that $E^i$ is connected, the
invariant need not be constant. Taking the $d$-th root stack along $E^i$
affects the relative generic order along $E^i$ by multiplying it with $d$.
Thus a necessary condition to obtain a factorisation as in
Definition~\ref{def-destackification} is that the relative generic order is
constant along all components of $\bs{E}$. Given that the independency index
vanishes everywhere, this condition is also sufficient.

\begin{prop}
\label{prop-destackified}
Let $(X, \bs{E})/S$ be a divisorial stack over a quasi-compact scheme $S$, and
assume that the independency index is everywhere zero at $X$. Let
$\pi\colon X \to X_\coarse$ be the coarse space. Then the pair
$(X_\coarse, \bs{E}_\coarse)/S$ satisfies the standard assumptions. In
particular, the stack $X_\coarse$ is smooth and $\bs{E}_\coarse$ has simple normal
crossings only.

If, in addition, the relative generic order is constant along each of the components
in $\bs{E}$, then the following holds:
\begin{enumerate}
\item The divisor $\bs{E}$ is the $\bs{d}$-th root of $\pi^\ast \bs{E}_\coarse$ for
some sequence $\bs{d}$ of positive integers indexed by the components of
$\bs{E}$.
\item The canonical factorisation
$
X \to (X_\coarse)_{\bs{d}^{-1}E_\coarse} \to X_\coarse
$
makes $X$ a gerbe over the stack $(X_\coarse)_{\bs{d}^{-1}E_\coarse}$.
\end{enumerate} 
\end{prop}
\begin{proof}
Let $\xi$ be a point at $X$, and let $A = A(\xi)$ be the character group of the
stabiliser at $\xi$. Choose local homogeneous coordinates
$R = \sheaf{O}_S[x_1, \ldots, x_r]$,
compatible with $\bs{E}$, with $x_i$ having degree $a_i \in A$. As in the beginning of
Section~\ref{sec-alg-outline}, we choose the indexing such that coordinates
$x_i$ for $1 \leq i \leq s$ correspond to components of $\bs{E}$, and we let
$A_\div = \langle a_1, \ldots, a_s\rangle$. By the assumption that the divisorial index is
zero, the degrees of the coordinates $x_i$ for $i > s$ are zero, and by the assumption
that the independency index is everywhere zero, we have $A_\div = \oplus_{i = 1}^s A_i$,
where $A_i$ is a finite cyclic group generated by $a_i$. Denote the order of $A_i$
by $d_i$. Then the coarse space is the relative spectrum of the invariant ring 
$$
R_0 = \sheaf{O}_S[x_1^{d_1}, \ldots, x_s^{d_s}, x_{s+1}, \ldots, x_r].
$$
In particular, the coarse space is smooth since this is a polynomial ring. The coarse space
$E^i_\coarse$ of a component $E^i$ of $\bs{E}$ corresponding to the $i$-th coordinate is $V(x_i^{d_i})$.
Hence also $E^i_\coarse$ is smooth, and the set $E^1_\coarse, \ldots, E^s_\coarse$ have simple normal
crossings only, which proves that $(X_\coarse, \bs{E}_\coarse)$ is a standard pair.

From the coordinates, we also see that $E^i$ is a $d_i$-th root
of $\pi^{-1}(E^i_\coarse)$ near $\xi$. Since $d_i$ is also the relative generic order along $E^i$,
and this assumed to be globally constant, the divisor $E^i$ is a $d_i$-th root of $\pi^{-1}(E^i_\coarse)$
globally, which proves (1). This gives the factorisation in (2) by the universal property of root
stacks. The fact that $X$ is a gerbe over $(X_\coarse)_{\mathbf{d}^{-1}E_\coarse}$ can be verified locally.
It follows from the sequence of homomorphisms between the local homogeneous coordinate rings of
graded $\sheaf{O}_S$-algebras
$$
(R_0, 0) \to (R, A_\div) \to (R, A).
$$
here the first map corresponds to the root stack and the second map corresponds to the gerbe.
\end{proof}

According to the discussion before the last proposition, we may need to spit the components
of the divisor $\bs{E}$ into smaller components after running Algorithm~\ref{alg-destack}. This can also
easily be described by giving an algorithm, fitting into the framework with conormal invariants and
blow-up sequences, which produces a sequence of trivial blow-ups, but we omit the details.
This concludes the proof of Theorem~\ref{theoremDestack}.

\appendix
\section{Tame stacks}
\label{appendix-tame}
This appendix may be viewed as a supplement to Section~3 of the article \cite{aov2008}.
We start by recalling some of the main concepts. Let $S$ be a scheme and $X$
an algebraic stack which is quasi-separated and locally of finite presentation
over $S$. If $X$ has finite inertia, there exists a coarse space
$\pi\colon X \to X_\coarse$ with the map $\pi$ being proper. Following
\cite{aov2008}, we say that $X$ is \term{tame} if the functor
$\pi_\ast \QCoh X \to \QCoh X_\coarse$ is exact. We call a group scheme
$G \to S$ \term{finite, linearly reductive} if $G$ is finite, flat, locally
of finite presentation and the fibres are linearly reductive. We say that
an algebraic stack $X$ has linearly reductive stabiliser at a point
$\xi \in |X|$ if the stabiliser at one of, or equivalently any of, the
$k$-points representing $\xi$ is linearly reductive.

The following theorem is an extension of the main theorem of \cite{aov2008}.
\begin{theorem}
\label{thm-tame}
Let $S$ be a scheme, and $X$ an algebraic stack which is quasi-separated
and locally of finite presentation over $S$. Assume that $X$ has finite inertia.
Then the following conditions are equivalent.
\begin{enumerate}
\item[(a)]
The stack $X$ is tame.
\item[(b)]
The stabilisers of $X$ are linearly reductive.
\item[(c)]
There exists a covering $Y \to X_\coarse$ of the coarse space, which is faithfully flat
and locally of finite presentation, a finite, linearly reductive group scheme $G \to Y$,
and a $G$-space $U \to Y$ which is finite and finitely presented, together with an isomorphism
$$
[U/G] \simeq Y \times_{X_\coarse} X.
$$
\item[(d)]
The same as (c), but $Y \to X_\coarse$ can be assumed to be étale.
\end{enumerate}
If, in addition, the morphism $X \to S$ is assumed to be smooth, the above conditions are
equivalent to the following condition.
\begin{enumerate}
\item[(e)] The same as (d), but $U$ can be assumed to be smooth over $S$.
\end{enumerate}
\end{theorem}
\noindent
The equivalence of the conditions (a)--(d) is \cite[Theorem~3.2]{aov2008}. Here,
we will prove that (e) is equivalent to the other conditions under the extra
hypothesis, and give a simplification of the proof that (b) implies (d).

In \cite{aov2008}, it is proven that tame gerbes admit sections étale locally.
The argument given is based on rigidification and the structure theory of linearly
reductive groups. But rather interestingly, the existence of an étale local section
is a consequence of a much more elementary fact regarding gerbes in general.
\begin{prop}
\label{prop-gerbe-smooth}
Let $S$ be a scheme and $X$ an algebraic stack which is an fppf gerbe over $S$.
Then the structure morphism $\pi\colon X \to S$ is smooth.
\end{prop}
\begin{proof}
The question is local on the base in the fppf topology, so we may assume that
$X$ is a classifying stack $\BB_S G$ for some group algebraic space $G$ which
is flat and locally of finite presentation over $S$. In particular,
we have an atlas $S \to X$. Let $U \to X$ be a smooth
atlas. Then the fibre product $U' = S \times_X U$ is an algebraic space
which is smooth over $S$, and the projection $U' \to U$ is faithfully flat and
locally of finite presentation. Hence $U$ is also smooth over $S$ by
\cite[Prop.\ 17.7.7]{egaIV4}, and it follows that $X$ is smooth over $S$.
\end{proof}
In fact, from the proof we see that the structure morphism of a gerbe has
all properties which are fppf local on the base and which {\em descend}
fppf locally on the source. Note, however, that although for instance being
étale is such a property when we restrict to morphisms of schemes, this is
not the case when we consider morphisms of algebraic stacks. Indeed, the
classifying stack $\BB\mu_p$ is not étale over the base if the base is
a field of characteristic~$p$.

One of the fundamental properties of finite linearly reductive groups acting on
algebraic spaces is that taking quotients of invariant closed subspaces
coincides with taking schematic images. We give a formulation of this property
in terms of tame stacks.
\begin{prop}
\label{prop-schematic-coarse}
Let $X \to S$ be a tame stack over a scheme $S$, and let $\pi\colon X \to
X_\coarse$ be the coarse space. Let $Z \subset X$ be a closed substack.
Then the canonical map $Z \to \pi(Z)$ to the schematic image of $Z$ through
$\pi$, is the coarse space of $Z$.
\end{prop}
\begin{proof}
The question may be verified after a faithfully flat base change of the coarse
space. Thus, we may use Theorem~\ref{thm-tame}~(c) to reduce to the case when
$X_\coarse = \Spec A$ for some ring $A$ and $X = [U/G]$, where $G$ is a linearly
reductive group scheme over $\Spec A$, and $U = \Spec B$, where $B$ is a finite
$A$-algebra. Let $I \subset B$ be the $G$-invariant ideal defining $Z$. Then the
coarse space of $Z$ is $\Spec (A/I)^G$ and the schematic image is given by
$\Spec A^G/I^G$. But the functor $-^G$ is exact since the group $G$ is linearly
reductive, so $(A/I)^G = A^G/I^G$ as desired.
\end{proof}

Proposition~\ref{prop-gerbe-smooth} and~\ref{prop-schematic-coarse} together imply
the following corollary, which is a reformulation of \cite[Prop.~3.7]{aov2008}.
This gives the simplification of the proof that (b) implies (d) in
Theorem~\ref{thm-tame}, which was promised earlier.
\begin{corollary}
Let $X \to S$ be a tame stack over a scheme $S$, and let $\pi\colon X \to
X_\coarse$ be the coarse space. Then the residual field at each point
$\xi \in |X|$ coincides with the residual field of the point $\pi(\xi)$. In
particular, every $k$-point in $X_\coarse$, with $k$ a field, lifts to
a $K$-point of $X$ for some separable field extension $K/k$.
\end{corollary}
\begin{proof}
Let $k = \kappa(\pi(\xi))$ be the residue field of $\pi(\xi)$. Let
$X_k \to \Spec k$ be the pull-back of $\pi$ along $\Spec k \hookrightarrow
X_\coarse$. Then $X_k \to \Spec k$ is a coarse space by
\cite[Cor.\ 3.3 (a)]{aov2008}. Furthermore, the induced map
$X_k^\red \to \Spec k$ from the reduction of $X_k$ is a coarse space by
Proposition~\ref{prop-schematic-coarse}, since the schematic image of $X_k^\red$
in $\Spec k$ must be $\Spec k$ itself. But the monomorphism $X_k^\red
\hookrightarrow X$, being a monomorphism from a reduced, locally noetherian
singleton, is the residual gerbe at $\xi$. Hence $k$ is indeed the residue field
at $\xi$. By Proposition~\ref{prop-gerbe-smooth}, the map $X_k^\red \to \Spec k$
is smooth, so it admits a section étale locally on $\Spec k$. From this, the
last statement of the proposition follows.
\end{proof}

Before turning to condition (e) of Theorem~\ref{thm-tame}, we review
what is meant by a \term{fixed point} for an action of by an algebraic group
$G$. Note that it is insufficient to just study the action of the topological
group $|G|$, as is illustrated by the following basic example.

\begin{example}
The group $\mu_2$ has a natural action on $\Spec \CC$ over $\Spec \RR$.
Topologically, the space $\Spec \CC$ has a single point, but it is not
accurate to say that the point is fixed under the $\mu_2$-action.
Rather, we wish to think of $\mu_2$ as acting freely on $\Spec \CC$,
making $\Spec \CC$ a $\mu_2$-torsor over $\Spec \RR$. In this case, we can
think of $\Spec \CC$ having two different geometric points $\Spec \CC \to \Spec
\CC$ over $\Spec \RR$, none of which is fixed under the $\mu_2$-action.

If we are working with non-reduced group schemes, even this point of view
does not work. This can be seen by instead considering the corresponding
example over the function field $k = \GF{p}(X)$, and the group $\mu_p$
acting on $\Spec k[Y]/(Y^p - X)$ over $\Spec k$.
\end{example}

Instead we consider the correct, sheaf-theoretic definition based on the
following proposition. We omit the proof since it is an easy diagram chase.
\begin{prop}
\label{prop-fixed-point}
Let $R\rightrightarrows U$ be a groupoid of sheaves on a site and let $\xi\colon
T \to U$ be a generalised point. Then the following statements are equivalent:
\begin{enumerate}
\item For any morphism $T' \to T$, the restriction of $\xi$ to $T'$ is the
unique representative of its isomorphism class in the groupoid
$R(T') \rightrightarrows U(T')$ viewed as a small category.
\item The graph $\Gamma_\xi\colon T \hookrightarrow U\times T$ is invariant with
respect to the groupoid $R\times T \rightrightarrows U \times T$.
\end{enumerate}
If the groupoid $R\rightrightarrows U$ is an action groupoid for a group action
$G\times U \to U$, then the above two statements are equivalent to the following
statement: 
\begin{enumerate}
\setcounter{enumi}{2}
\item The canonical monomorphism $\Stab(\xi) \to G\times T$ of groups over $T$
is an isomorphism.
\end{enumerate} 
\end{prop}
\begin{definition}
Let $R\rightrightarrows U$ and $\xi\colon T \to U$ be as in
Proposition~\ref{prop-fixed-point}. If the conditions given in the proposition
are satisfied, we say that $\xi$ is a \term{fixed point} for the groupoid
$R\rightrightarrows U$. If the groupoid is algebraic and $\xi \in |U|$ is a point
in $U$, we say that $\xi$ is a \term{fixed point} if it may be represented by
a morphism $\Spec k \to U$ which is a fixed point in the above sense. It is
easily verified that the choice of representative is irrelevant.
\end{definition}

That condition (b) implies condition (d) of Theorem~\ref{thm-tame}, follows from
the sharper Proposition~3.6 of \cite{aov2008}. In order to see that it also
implies (e), we need to sharpen the formulation of the proposition somewhat
more.
\begin{prop}
\label{prop-tame-fixed}
Let $S$ be a scheme and $X$ an algebraic stack having finite inertia and
being quasi-separated and locally of finite presentation over $S$. Denote
the coarse space by $\pi\colon X \to X_\coarse$, and let $\xi\in |X|$ be a point.
If the stabiliser at $\xi$ is linearly reductive, then there exists an étale
neighbourhood $Y \to X_\coarse$ of $\pi(\xi)$, a finite, linearly reductive group
scheme $G \to Y$ acting on a finite scheme $U \to Y$ of finite presentation, and
an isomorphism $[U/G] \simeq Y \times_{X_\coarse}X$ of algebraic stacks.
Furthermore, the point $\xi$ lifts to a point $\xi' \in U$ which is fixed under
the action of $G$.
\end{prop}
All but the last sentence comes from the original statement, and although the
last sentence is not explicitly stated, it follows from the proof. Indeed, the
scheme $U \to Y$ is constructed in a way such that the diagram
$$
\xymatrix{
\Spec k \ar[r]^{\xi'}\ar[d] & U \ar[d]\\
\BB_kG_\xi \ar[r] & Y \times_{X_\coarse}X\\
}
$$
becomes cartesian. Here $\xi\colon \Spec k \to X$ is a morphism
representing $\xi$, the vertical maps are $G$-torsors, and $G_\xi$ denotes the
stabiliser at $\xi$. In particular, the point $\xi'$ becomes the desired
lifting according to the third condition of Proposition~\ref{prop-fixed-point}
characterising fixed points. Finally, to see that this implies (e) in the
case when $X$ is smooth over the base, we apply the following proposition.

\begin{prop}
\label{prop-fixed-smooth}
Let $U$ be an algebraic space which is flat, locally of finite presentation
and quasi-separated over a scheme $S$. Assume that $R \rightrightarrows U$ is
a groupoid which is flat and locally of finite presentation, and assume that
the stack quotient $[U/R]$ is smooth over $S$. Then $U$ is smooth over $S$ at
any point $\xi \in |U|$ which is a fixed point with respect to the
groupoid $R \rightrightarrows U$.
\end{prop}
\begin{proof}
Let $\xi\colon\Spec k \to U$ be a geometric point representing $\xi$, and let
$R_k \rightrightarrows U_k$ denote the pull-back of the groupoid along the
morphism $\Spec k \to U \to S$. Since $\xi$ is a fixed point, the graph
$\Gamma_\xi\colon \Spec k \to U_k$ is invariant in the groupoid.
Hence, the diagram 
$$
\xymatrix{
\Spec k \ar[r]^{\Gamma_\xi} \ar[d] & U_k \ar[d]\\
\BB_k \Stab(\xi) \ar[r]_{\iota} & [U_k/R_k]
}
$$
is 2-cartesian. The graph $\Gamma_\xi\colon \Spec k \to U_k$ is a closed
immersion since it is a rational point. Since the vertical maps in the diagram are faithfully
flat and locally of finite presentation, it follows that also $\iota$ is a closed
immersion, by descent. The stack $[U_k/R_k]$ is smooth
over $k$ since it is isomorphic to the pullback $[U/R]\times_S \Spec k$ and smoothness
is stable under base change. The stack $\BB_k \Stab(\xi)$ is smooth over $k$ since
it is a gerbe. It follows that $\iota$ is a regular immersion, so the same holds
for the graph $\Gamma_\xi$, since the property of being a regular immersion is
stable under flat base change. But then $U_k$ must be regular at $\Gamma_\xi$.
Since $U$ is flat and of finite presentation over $S$, it follows that $U \to S$
is smooth at $\xi$.
\end{proof}

We conclude the section with a technical lemma, which is not related to tame
stacks, about closed points on stacks. In general, rational points on algebraic
stacks need not be closed. For instance, the stack $[\AA^1_k/\GG_m]$ has an open
rational point. But stacks with finite stabilisers are better behaved.

\begin{lemma}
\label{rational-point-closed}
Let $k$ be a field and $X$ an algebraic stack which is locally of finite type
and quasi-separated over $\Spec k$. If $X$ has finite stabilisers, then every
point of finite type in $X$ is closed. In particular, all rational points of
$X$ are closed.
\end{lemma}
\begin{proof}
Let $\xi \in |X|$ be a point of finite type in $X$, and let
$f\colon \gerbe{G}_\xi \hookrightarrow X$ be the inclusion of the residual
gerbe at $\xi$. By the assumption that $\xi$ is a point of finite type, the
monomorphism $f$ is locally of finite type. We want to show that $f$ is a
closed immersion.

We may assume that $X$ is of finite type over $k$. Since $X$ is quasi-separated
and has finite stabilisers, we can choose a quasi-finite, flat covering $U \to
X$ \cite[Thm.~7.1]{rydh2011}. Let $U_\xi = U \times_X \gerbe{G}_\xi$. Since
$U_\xi$ is quasi-finite over the residue field, it is a scheme with a finite
discrete underlying topological space. In particular,
each open subscheme of $U_\xi$ is affine, so $U_\xi \to U$ is a finite monomorphism,
and therefore a closed immersion.
\end{proof}

\section{The cotangent complex of toric stacks}
\label{appendix-cotangent}
We wish to compute the cotangent complex for a basic toric stack $X$ without
torus factors over a scheme $S$. Assume that $X$ has homogeneous coordinates
$(\sheaf{O}_S[x_1, \ldots, x_n], A, \bs{a})$ where $A = \ZZ/q_1\ZZ\times \cdots \times \ZZ/q_s\ZZ$
and $a_i = (a_{1i}, \ldots, a_{si})$ for $1 \leq i \leq n$.

The quasi-coherent $\sheaf{O}_X$-modules are in canonical one-to-one correspondence
with the quasi-coherent $A$-graded $\sheaf{O}_S[x_1, \ldots, x_n]$-modules. Given
$a\in A$, we denote by $\sheaf{O}_X(-a)$ the line bundle corresponding to the
free $\sheaf{O}_S[x_1, \ldots, x_n]$-module of rank one generated in
degree $a$.

\begin{prop}
Let $X$ be a basic toric stack over a scheme $S$, and assume that $X$ has homogeneous
coordinates as described above. Then the cotangent complex $L_{X/S}$ is quasi-isomorphic
to the perfect complex
$$
\sheaf{O}_X(-a_1)\oplus \cdots \oplus \sheaf{O}_X(-a_n)
\oplus \sheaf{O}_X^s \to \sheaf{O}_X^s
$$
concentrated in cohomological degrees $[0, 1]$ and with differential given by the
matrix
$$
\left(
\begin{matrix}
a_{11}x_1 & \cdots & a_{1n}x_n & q_1 & & 0 \\
\vdots & \ddots & \vdots & & \ddots & \\
a_{s1}x_1 & \cdots & a_{sn}x_n & 0 & & q_s \\
\end{matrix}
\right).
$$
\end{prop}
\begin{proof}
Due to the base change properties of the cotangent complex, we may just as well assume
that $S = \Spec \ZZ$. Consider the  $\ZZ^s$-graded coordinate ring
$\sheaf{O}_{\AA^n\times \GGm^s} = \ZZ[x_1, \ldots, x_n,t^{\pm1}_1, \ldots, t^{\pm 1}_s]$
of the space $\AA^n\times \GGm^s$, with $x^i$ having degree $a_i$, viewed as a
vector of integers, and $t_j$ having degree $q_j$. The grading corresponds to
an action of the torus $\GGm^s$. The stack quotient
$[\AA^n\times \GG_m^s /\GGm^s]$ is equivalent to $X$, with the equivalence
induced by slicing the action groupoid
$\AA^n\times\GG_m^s\times\GG_m^s\rightrightarrows \AA^n\times\GG_m^s$ at the
closed subscheme $V(t_j = 1 \mid 1\leq j\leq s)$ of $\AA^n\times\GG_m^s$.
Indeed, then we get the Morita equivalent groupoid $\AA^n\times\Delta
\rightrightarrows \AA^n$, with $\Delta = A^\vee$, which is the presentation of
$X$ corresponding to the original $A$-grading.

The atlas introduced above gives us a cartesian square
$$
\xymatrix {
\AA^n\times\GGm^s\times\GGm^s \ar[d]_\pi \ar[r]^-\alpha & \AA^n\times\GGm^s
\ar[d]_q\\
\AA^n\times\GGm^s \ar[r]_-q & X \\
}
$$
of smooth morphisms. Here $\pi$ denotes the projection on the first two factors and
$\alpha$ denotes the action map. Choose coordinates 
$\sheaf{O}_{\AA^n\times \GGm^s\times\GGm^s} = \ZZ[x_i, t^{\pm1}_j, u^{\pm1}_k]$.
Then the action map corresponds to the ring homomorphism $\varphi$ taking $x_i$
to $x_iu_1^{a_{1i}}\cdots u_s^{a_{si}}$ and $t_i$ to $t_iu_i^{q_i}$.
We get an induced map of differentials
$$
d\varphi\colon\alpha^*\Omega_{\AA^n\times \GGm^s} \to
\Omega_{\AA^n\times \GGm^s\times\GGm^s/\AA^n\times \GGm^s}
$$
given by
$$
dx_i \mapsto a_{1i}\varphi(x_i)u_1^{-1}du_1 + \ldots + a_{si}\varphi(x_i)u_s^{-1}du_s, \qquad
dt_j \mapsto q_j\varphi(t_j)u_j^{-1}du_j.
$$
Since the above square is cartesian and $\alpha$ is flat, this descends to a map
$\Omega_{\AA^n\times\GG_m} \to \Omega_{\AA^n\times\GG_m/X}$.
We choose $(dx_i, dt_j)$ as a basis for the left hand side. An easy
calculation gives that the elements $u_i^{-1}du$ descend to elements of
$\Omega_{\AA^n\times \GGm/X}$, which we also denote by $u_i^{-1}du$. These
elements form a basis for the right hand side. With respect to these choices,
the map is described by the matrix
$$
\left(
\begin{matrix}
a_{11}x_1 & \cdots & a_{1n}x_n & t_1q_1 & & 0 \\
\vdots & \ddots & \vdots & & \ddots & \\
a_{s1}x_1 & \cdots & a_{sn}x_n & 0 & & t_sq_s \\
\end{matrix}
\right).
$$
Now we may compute the cotangent complex with help of the distinguished triangle
$$
\Omega_{\AA^n\times\GGm^s} \to \Omega_{\AA^n\times\GGm^s/X} \to
\Lpb{q}L_{X}[1].
$$
Hence the derived pullback $\Lpb{q}L_{X}$ is given by a two term complex
$\Omega_{\AA^n\times\GGm^s} \to \Omega_{\AA^n\times\GGm^s/X}[-1]$ with
differential as in the matrix above. This is graded by $\ZZ^s$, with $|dx_i| = a_i$, $|dt_j| = q_j$
and $|u^{-1}_kdu_k| = 0$, which reflects that the complex lives over $[\AA^n/\Delta]$.
We obtain the original homogeneous coordinates by slicing at $V(t_j = 1 \mid 1 \leq j \leq s)$
as described before, which gives the result stated in the proposition.
\end{proof}

Note that if the product $q_1\cdots q_s$ is invertible in $\sheaf{O}_S$, then
this complex is quasi-isomorphic to the $\sheaf{O}_{X}$-module
$$
\sHH^0(L_{X/S}) = \sheaf{O}_X(-a_1) \oplus \cdots \oplus \sheaf{O}_X(-a_n).
$$
The fibre of this module in a point $\xi$, together with the natural action of the
stabiliser, coincides with the conormal representation at $\xi$. If $q_1\cdots q_s$
is not invertible, this need not be true. Consider, for instance,
$X = [\AA^1_S/\mu_q]$ where $S = \Spec k$ with $k$ a field of characteristic $p$ and $q = 0$ in $k$.
If $\mu_q$ is acting with weight $a$ we have
$$
\sHH^0(L_{X/S}) =
\left\{
\begin{array}{ll}
\sheaf{O}_X(-a)\oplus \sheaf{O}_X & \text{if } a = 0 \text{ in } k, \\
\sheaf{O}_X & \text{otherwise.}
\end{array}
\right.
$$
We see that, in general, the information about the weight is lost. If one
wishes to preserve this information, it is better to look at $[L_{X/S}]$, or
equivalently, the alternating sum $[\sHH^0(L_{X/S})] - [\sHH^1(L_{X/S})]$ in the
$K$-group.

\section{Cotangent complex interpretation}
\label{appendix-cotangent-conormal}
In this appendix, we take a brief look at an alternative way to look
at the conormal representation in terms of the cotangent complex.
An advantage with this point of view is that we can use various distinguished
triangles for the cotangent complex in our computations.

Given an algebraic stack $X$, we consider the triangulated category
$\perf{X}$ of perfect complexes, and its associated Grothendieck group
$\KK_0(\perf{X})$.
A morphism $f\colon X \to Y$ gives a morphism $f^\ast\KK_0(\perf{Y}) \to
\KK_0(\perf{X})$ induced by the derived pull-back.
If $X$ is smooth over a field, then $\KK_0(\perf{X})$ is canonically isomorphic
to $\KK_0(\coh{X})$.

In the particular situation described in the beginning of Section~\ref{sec-conormal},
we have a 2-commutative diagram
$$
\xymatrix{
\BB\Delta_\xi \ar[r]^\iota \ar[rd] & X_{\bar{k}} \ar[r]^f \ar[d] & X \ar[d]\\
  & \Spec\bar{k} \ar[r] & S
}
$$
where the square is 2-cartesian. Denote the composition $f\circ\iota$ by
$g$. By using the distinguished triangle for composition and the base change
property for cotangent complexes, we get the identities
$$
\iota^\ast[L_{X_{\bar{k}}/\bar{k}}]
- [L_{\BB\Delta_\xi/\bar{k}}]
+ [L_{\BB\Delta_\xi/X_{\bar{k}}}] = 0,
\qquad
[L_{X_{\bar{k}}/\bar{k}}] = f^\ast[L_{X/S}]
$$
in $\KK_0(\perf{\BB\Delta_\xi})$. Since the immersion $\BB\Delta_\xi
\hookrightarrow X_{\bar{k}}$ is regular, the cotangent complex
$L_{\BB\Delta_\xi/X_{\bar{k}}}$ is quasi-isomorphic to the complex having
$\sheaf{N}_{\BB\Delta_\xi/X_{\bar{k}}}$ concentrated in degree $-1$.
Together with the identities above, this implies that
$$
[\sheaf{N}_{\BB\Delta_\xi/X_{\bar{k}}}]
= -[L_{\BB\Delta_\xi/X_{\bar{k}}}]
= g^\ast[L_{X/S}] - [L_{\BB\Delta_\xi/\bar{k}}]
$$
in $\KK_0(\perf{\BB\Delta_\xi})$. If $\bar{k}$ has characteristic 0, the complex
$L_{\BB\Delta_\xi/\bar{k}}$ vanishes since $\BB\Delta_\xi$ is étale over
$\bar{k}$. In positive characteristic, the complex $L_{\BB\Delta_\xi/\bar{k}}$
need not vanish, but its class in the Grothendieck group always vanishes. This
can be seen from the explicit formula derived in Appendix~\ref{appendix-cotangent}.
In particular, we have the identity $[\sheaf{N}_{\BB\Delta_\xi/X_{\bar{k}}}] =
g^\ast[L_{X/S}]$. We summarise the result in the following proposition.

\begin{prop}
Let $(X, \bs{E})/S$ be a standard pair with diagonalisable stabilisers.
Furthermore, we let $\xi\colon \Spec \bar{k} \to X$ be a geometric point,
$\Delta_\xi$ the stabiliser at $\xi$ and $g\colon \BB\Delta_\xi \to X$
the induced morphism. Then we have the identity
$$
[\sheaf{N}_{\BB\Delta_\xi/X_{\bar{k}}}] = g^\ast[L_{X/S}]
$$
in the group $\KK_0(\perf{\BB\Delta_\xi})$, which we have identified with
$\KK_0(\coh{\BB\Delta_\xi})$.
\end{prop}

\bibliographystyle{myalpha}
\bibliography{main}
\end{document}